\documentclass[11pt]{amsart}
\usepackage{fullpage}
\usepackage{amsmath, amssymb, amsthm}

\usepackage{graphicx}
\usepackage{color}
\usepackage[plainpages=false,colorlinks=true,citecolor=blue,filecolor=black,linkcolor=red]{hyperref}

\newtheorem{thm}{Theorem}[section]

\newtheorem{prop}[thm]{Proposition}
\newtheorem{lem}[thm]{Lemma}
\theoremstyle{remark}
\newtheorem{rem}[thm]{Remark}
\newtheorem{ex}[thm]{Example}

\newcommand{\ie}{{\it i.e.}}
\newcommand{\eg}{{\it e.g.}}

\begin{document}

\title{Computation of free boundary minimal surfaces \\ via extremal Steklov eigenvalue problems}

\author{Chiu-Yen Kao}
\address{Department of Mathematical Sciences, Claremont McKenna College, Claremont, CA 91711}
\email{ckao@cmc.edu}

\author{Braxton Osting}
\address{Department of Mathematics, University of Utah, Salt Lake City, UT}
\email{osting@math.utah.edu}

\author{\'Edouard Oudet}
\address{LJK, Universit\'e Grenoble Alpes, France}
\email{edouard.oudet@imag.fr}



\date{\today}

\begin{abstract} 
Recently Fraser and Schoen  showed that the solution of a certain extremal Steklov eigenvalue problem on a compact surface with boundary can be used to generate a free boundary minimal surface, i.e., a surface contained in the ball that has (i) zero mean curvature and  (ii) meets the boundary of the ball orthogonally (doi:\href{https://doi.org/10.1007/s00222-015-0604-x}{10.1007/s00222-015-0604-x}). 
In this paper, we develop numerical methods that use this connection to realize free boundary minimal surfaces. 
Namely, on a compact surface, $\Sigma$, with genus $\gamma$ and  $b$ boundary components, 
we  maximize $\sigma_j(\Sigma,g) \ L(\partial \Sigma, g)$ over a class of smooth metrics, $g$, 
where $\sigma_j(\Sigma,g)$  is the $j$-th nonzero Steklov eigenvalue and 
$L(\partial \Sigma, g)$ is the length of $\partial \Sigma$. 
Our numerical method involves 
(i) using conformal uniformization of multiply connected domains to avoid explicit parameterization for the class of metrics,  
(ii) accurately solving a boundary-weighted Steklov eigenvalue problem in multi-connected domains, and
(iii) developing gradient-based optimization methods for this non-smooth eigenvalue optimization problem. 
For genus $\gamma =0$ and $b=2,\dots, 9, 12, 15, 20$ boundary components, we numerically solve the extremal Steklov problem for the first eigenvalue. 
The corresponding eigenfunctions generate a free boundary minimal surface, which we display in striking images. 
For higher eigenvalues, numerical evidence suggests that the maximizers are degenerate, 
but we compute local maximizers 
 for the second and third eigenvalues with $b=2$ boundary components
 and for the third and fifth eigenvalues with $b=3$ boundary components. 
\end{abstract}

\maketitle

\section{Introduction}
Recently, A. Fraser and R. Schoen discovered a rather surprising connection between 
an extremal Steklov eigenvalue problem and 
the problem of generating free boundary minimal surfaces in the Euclidean ball \cite{Fraser_2011,Fraser_2013,Fraser_2015}. 
These findings have been further developed \cite{Fan_2014,Fraser2019,Girouard2020} and were recently reviewed in \cite{Li2019}. 
In this paper, we develop numerical methods to further investigate this connection. 
We first briefly review some of these previous results before stating the contributions of the present work.

\subsection*{The extremal Steklov eigenvalue problem}
Let  $(\Sigma,g)$ be a smooth, compact, connected Riemannian surface with nonempty boundary, $\partial \Sigma$. 
The Steklov eigenproblem on  $(\Sigma,g)$ is given by 
\begin{subequations}
\label{e:Steklov}
\begin{align}
& \Delta v =0 && \Sigma \\
& \partial_\nu v = \sigma v && \partial \Sigma,
\end{align}
\end{subequations}
where 
$\Delta  = |g|^{- \frac 1 2} \partial_i |g|^{\frac 1 2} g^{i j} \partial_j $ is the Laplace-Beltrami operator and 
$ \partial_\nu$ is the outward normal derivative. 
The Steklov spectrum is discrete and  we enumerate the eigenvalues, counting multiplicity,  in increasing order
$$
0 \ = \  \sigma_0(\Sigma,g)  \ <  \ \sigma_1(\Sigma,g) \ \leq \  \sigma_2(\Sigma,g) \ \leq \ \cdots \  \to \infty.
$$ 
The Steklov spectrum coincides with the spectrum of the Dirichlet-to-Neumann operator $\Gamma \colon H^{\frac1 2}(\partial \Sigma) \to H^{- \frac 1 2}(\partial \Sigma)$,  given by the formula $\Gamma w =  \partial_\nu ( \mathcal H w)$, where $\mathcal H w$ denotes the unique harmonic extension of $ w \in H^{\frac1 2}(\partial \Sigma)$ to $\Sigma$.
The restriction of the Steklov eigenfunctions to the boundary, 
$\{ v_j |_{\partial \Sigma} \}_{j=0}^\infty \subset C^\infty(\partial \Sigma)$, 
form a complete orthonormal basis of $L^2(\partial \Sigma)$. 
A recent survey on Steklov eigenvalues can be found in  \cite{Girouard_2017}.

Here, for fixed surface $\Sigma$ with genus $\gamma$ and $b$ boundary components, we consider the dependence of the $j$-th Steklov eigenvalues on the metric, \ie, the mapping $g\to \sigma_j(\Sigma,g)$.  
It is known that for any smooth Riemannian metric $g$, we have the following upper bound on the $j$-th Steklov eigenvalue in terms of the topological invariants  $\gamma$ and $b$, 
\begin{equation} \label{e:Isoperimetric}
\sigma_j(\Sigma, g) \  L(\partial \Sigma,g) \ \leq \ 2 \pi (\gamma + b + j-1 )
\qquad \qquad 
\forall j \in \mathbb N. 
\end{equation}
 Here, $L(\partial  \Sigma, g)$ is the length of $\partial \Sigma$ with respect to the metric $g$.
 This bound was proven
 by Weinstock  \cite{Weinstock_1954} for $j=1$, $\gamma =0$, and $b=1$; 
 by  Fraser and Schoen \cite{Fraser_2011} for $j=1$ (see also \cite{Girouard_2012}); 
 and in generality by Karpukhin \cite{Karpukhin_2017}.   
 It is then natural to pose the \emph{extremal Steklov eigenvalue problem}, 
\begin{equation}
\label{e:MaxSteklov}
\tilde \sigma_j^\star(\gamma, b)  := \sup_g \ \tilde \sigma_j(\Sigma, g), 
\qquad \qquad 
\tilde \sigma_j(\Sigma, g) := \sigma_j(\Sigma, g) \  L(\partial  \Sigma,g),
\end{equation}
where $g$ varies over the class of smooth Riemannian metrics on $\Sigma$. The existence of a smooth maximizer in \eqref{e:MaxSteklov} was established 
 in \cite[Theorem 1.1]{Fraser_2015}
for oriented surfaces of genus $0$ with $b\geq 2$ boundary components  or   a M\"obius band  
and  in \cite{Matthiesen2020} for general surfaces for the first ($j=1$) eigenvalue. 
 
\subsection*{Free boundary minimal surfaces}
 Denote the closed $n$-dimensional Euclidean unit ball by 
$\mathbb B^n := \{ x \in \mathbb R^n \colon |x| \leq 1 \}$ and the 
$(n-1)$-dimensional unit sphere by $\mathbb S^{n-1} = \partial \mathbb B^n$. 
Let $\mathcal M \subset \mathbb B^n$ be a $d$-dimensional submanifold with boundary $\partial \mathcal M = \overline{\mathcal M} \cap \mathbb S^{n-1} $. We say that $\mathcal M$ is a \emph{free boundary minimal submanifold in the unit ball} if
\begin{enumerate}
\item[(i)] $\mathcal M$ has zero mean curvature and 
\item[(ii)] $\mathcal M$ meets $\mathbb S^{n-1}$ orthogonally along $\partial \mathcal M$. 
\end{enumerate}
When $d = 2$, we call $\mathcal M$ a  \emph{free boundary minimal surface in the unit ball}  or, more simply, a \emph{free boundary minimal surface}. For a good visual aid to understanding the definition of free boundary minimal surfaces (and a peak at the results of this paper), 
we recommend the reader take a look at the free boundary minimal surfaces displayed in Figures~\ref{f:surf34} and \ref{f:surf51215}.

\subsection*{Fraser and Schoen's connection} Fraser and Schoen observed that a $d$-dimensional submanifold $\mathcal M \subset \mathbb B^n$   with boundary $\partial \mathcal M = \overline{\mathcal M} \cap \mathbb S^{n-1} $ is a  free boundary minimal surface if and only if the coordinate functions $x_i$, $i = 1, \ldots, n$ restricted to $\mathcal M$ are Steklov eigenfunctions with eigenvalue $\sigma = 1$. Furthermore, they showed the following theorem. 

\begin{thm}[\cite{Fraser_2013}]  \label{t:Fraser}
Let $\Sigma$ be a compact surface with boundary. Suppose that $g_0$ is a smooth metric on $\Sigma$ attaining the supremum in \eqref{e:MaxSteklov} for  some $j \in \mathbb N$. 
Let $U$ be the $n$-dimensional eigenspace corresponding to $\sigma_j(\Sigma,g_0)$. 
Then, there exist independent Steklov eigenfunctions $u_1, \ldots, u_n \in U$ which give a (possibly branched) conformal immersion $u = (u_1,\cdots, u_n)\colon \Sigma \to \mathbb B^n$ such that $u(\Sigma)$ is a  free boundary minimal surface in $\mathbb B^n$ and, up to rescaling of the metric,  $u$ is an isometry on $\partial \Sigma$.
\end{thm}
Theorem~\ref{t:Fraser} gives a method for using the solution of \eqref{e:MaxSteklov}  to compute free boundary minimal surfaces. The simplest such example is the equatorial disk, obtained as the intersection of  $\mathbb B ^3$ with any two-dimensional subspace of $\mathbb R^3$. This can be constructed from Weinstock's result that inequality in \eqref{e:Isoperimetric} with $j=1$, $\gamma =0$, and $b=1$ is attained only by the round disk, $\mathbb D$   \cite{Weinstock_1954}. In this case, for the eigenvalue $\tilde \sigma_1(0,1) = 2\pi$, we have the two-dimensional eigenspace given by $\textrm{span} \{ x, y \}$. The equatorial disk is given as the map $u\colon \mathbb D  \to \mathbb R^2$, defined by 
$u(x,y) = \begin{pmatrix} x \\ y \end{pmatrix}$.

For genus $\gamma = 0$  and  $b =2$ boundary components, the extremal metric is rotationally invariant and the corresponding free boundary minimal surface is the critical catenoid. We will discuss this example further in Section~\ref{s:CC}. For genus $\gamma = 0$  and  $b \geq 3$ boundary components, the extremal metric is not known explicitly, but it is known that the corresponding  free boundary minimal surface is embedded in $\mathbb B^3$ and star-shaped with respect to the origin \cite{Fraser_2013}. 
In \cite{Girouard2020}, the authors used homogenization methods to construct surfaces that have large first Steklov eigenvalue $\tilde \sigma_1$. 
In particular, free boundary minimal surfaces of genus $\gamma = 0$ with particular symmetries (\eg, symmetries of platonic solids) were constructed numerically. 
The authors proved that the first nonzero Steklov eigenvalue, $\sigma_1$, of these surfaces is $1$ and emphasized that it is not known whether these surfaces have extremal first eigenvalues among all surfaces with the same genus and number of boundary components. We will compare our results to these surfaces in Section~\ref{s:NumResults}.

In \cite{Fan_2014}, Fan, Tam, and  Yu extended the study of \eqref{e:MaxSteklov} to higher values of $j$ on the cylinder ($\gamma = 0$, $b=2$) among rotationally symmetric conformal metrics. 
They obtained different results for even and odd eigenvalues. 
They showed that the maximum of the $\tilde \sigma_{2j -1}$, $j\in \mathbb N$ among all rotationally symmetric conformal metrics on the cylinder is achieved by the $j$-fold covering of the critical catenoid immersed in $\mathbb R^3$.
The maximum of $\tilde \sigma_2$ is not attained. 
The maximum of the $\tilde \sigma_{2j}$ for $j \geq 2$ among all rotationally symmetric conformal metrics on the cylinder is achieved by the $j$-fold covering of the critical M\"obius band.
These results will be further discussed in Section~\ref{s:CC} and further compared to our computed surfaces  in Section~\ref{s:NumResults}.

\subsection*{Results and outline} 
In this paper, we develop computational methods for solving the extremal Steklov eigenvalue problem \eqref{e:MaxSteklov} and thus generating free boundary minimal surfaces via Theorem~\ref{t:Fraser}. 
This approach is used to realize free boundary minimal surfaces beyond the known examples of 
equatorial disks, the critical catenoid, the critical M\"obius band, and their higher coverings discussed above. 

In Section~\ref{s:Red}, we explain how the conformal uniformization of multiply connected domains can be used to significantly reduce the complexity of the general Steklov eigenproblem \eqref{e:Steklov} and extremal Steklov eigenproblem \eqref{e:MaxSteklov}. The argument relies on two ingredients:
\begin{enumerate}
\item The uniformization result that for a smooth, compact, connected, genus-zero Riemannian surface with $b$ boundary components, $(\Sigma, g)$, there exists a conformal mapping $f \colon (\Sigma,g) \to (\Omega,\rho I)$, 
where $\Omega$ is a disk with $b-1$ holes and $\rho I$ is a conformally flat metric. 
\item The composition $v\circ f$ of a function $v$ with a conformal map $f$  is harmonic if and only if $v$ is harmonic. 
\end{enumerate}

Let $D =  \{ x \in \mathbb R^2 \colon |x| \leq 1\}$ be the unit disk and 
$$
\Omega_{c,r} = D \ \setminus \  \cup_{i=1}^{b-1} D_i 
$$
be a punctured unit disk with $b-1$ holes, 
$$
D_i = D(c_i,r_i) = \{x \in \mathbb R^2 \colon |x-c_i| < r_i \} 
\qquad \qquad 
i =1, \ldots , b-1.
$$ 
This argument implies that it is sufficient to consider the family of (flat!) Steklov eigenproblems, 
\begin{subequations}
\label{e:Steklov2}
\begin{align}
\label{e:Steklov2a}
& \Delta u =0 && \Omega_{c,r} \\
\label{e:Steklov2b}
& \partial_n u = \sigma \rho u && \partial \Omega_{c,r},
\end{align}
\end{subequations}
where 
$\Delta $ is the Laplacian on $\Omega$, 
$ \partial_n$ is the outward normal derivative, and 
$\rho>0$ is a density function.  The  extremal Steklov  eigenvalue problem \eqref{e:MaxSteklov} for genus $\gamma = 0$ is transformed to 
\begin{subequations}
\label{e:MaxSteklov2}
\begin{align}
\tilde \sigma_j^\star(\gamma=0, b)  = 
\max_{c_i, \ r_i, \ \rho} \ &   \tilde \sigma_j  \\
\textrm{s.t.} \ & D_i \subset D, && i = 1,\ldots, b-1 \\
& D_i \cap D_j = \varnothing,  &&  i\neq j \\
& \rho(x) \geq 0, && x \in \partial \Omega_{c,r}. 
\end{align}
\end{subequations}
Here, 
$\tilde \sigma_j = \sigma_j L$, 
$\sigma_j$ is the $j$-th nontrivial eigenvalue satisfying \eqref{e:Steklov2}, and 
$L = \int_{\partial \Omega_{c,r}} \rho(x) \ dx$
is the total length of $\partial \Omega_{c,r}$.
The first two constraints simply state that the holes are contained in the domain and are pairwise disjoint.

In Section~\ref{s:CC}, we explicitly solve the Steklov eigenvalue problem on a rotationally symmetric annulus (\ie, $\gamma = 0$, $b=2$, $c_1 = 0$, and $\rho$ constant on each boundary component) and describe the critical catenoid and its higher coverings in detail. These Steklov eigenvalues and corresponding  free boundary minimal surfaces will be used to verify our computational methods. 

In Section~\ref{s:CompMeth}, we develop numerical methods for 
computing Steklov eigenvalues satisfying \eqref{e:Steklov2} on multiply connected domains, 
computing the solution to the optimization problem \eqref{e:MaxSteklov2},  
and the computation of  free boundary minimal surfaces from the Steklov eigenfunctions.
 In brief, we use the method of particular solutions to compute Steklov eigenvalues, 
 gradient-based interior point methods for the optimization problem, and 
 compute the mapping to a surface by minimizing a particular energy. 
These  methods build on previous computational methods for extremal eigenvalue problems on Euclidean domains, 
including 
minimizing Laplace-Dirichlet eigenvalues over Euclidean domains of fixed volume or perimeter \cite{oudet2004numerical,osting2010optimization,antunes2012numerical,osting2013minimal,osting2014minimal,antunes2017numerical,Bogosel_2017}, 
maximizing Steklov eigenvalues over two-dimensional Euclidean domains of fixed volume \cite{Akhmetgaliyev_2017,Bogosel_2017}. 
These methods have recently been extended to more general geometric settings. 
In particular, \cite{Kao_2017}  maximized Laplace-Beltrami eigenvalues over conformal classes of metrics with fixed volume and compact Riemannian surfaces of fixed genus ($\gamma = 0$, 1) and volume.

In Section~\ref{s:NumResults}, we present the results of numerous computations. 
For genus $\gamma =0$ and $b=2,\dots, 9, 12, 15, 20$ boundary components, we numerically solve the extremal Steklov problem \eqref{e:MaxSteklov2} for the first eigenvalue. 
We include figures displaying the optimal punctured disks and 
three linearly-independent eigenfunctions associated to the first eigenvalue, as well as 
tabulate the values of the obtained Steklov eigenvalues. 
We also plot the associated free boundary minimal surfaces, which are visually striking. 
Finally, in Section~\ref{s:NumResults}, we also present results for maximizing higher eigenvalues. 
Here, numerical evidence suggests that the maximizers are degenerate, 
but we compute local maximizers 
 for the second and third eigenvalues with $b=2$ boundary components
 and for the third and fifth eigenvalues with $b=3$ boundary components. 
For brevity, we were only able to report the results for selected values of $b$ and $j$; 
the results of additional computations can be found on \'E. Oudet's website \cite{EdouardWebpage}, along with gifs. 

We conclude in Section~\ref{s:Disc} with a discussion. 


\section{The Euclidean Steklov eigenproblem} 
\label{s:Red}

In Section~\ref{s:uniformization}, we explain how the conformal uniformization of multiply connected domains can be used to significantly reduce the complexity of the general Steklov eigenproblem \eqref{e:Steklov} and extremal Steklov eigenvalue problem \eqref{e:MaxSteklov} to obtain the Euclidean Steklov eigenproblem and  \eqref{e:Steklov2} and  extremal Steklov eigenvalue problem  \eqref{e:MaxSteklov2}, respectively.
In Section~\ref{s:EigDeriv}, we also compute the eigenvalue derivatives with respect to the density and shape parameters and discuss optimality conditions for the extremal Steklov eigenvalue problem  \eqref{e:MaxSteklov2}.

\subsection{Conformal uniformization of multiply-connected surfaces and the Steklov eigenproblem} \label{s:uniformization}
The uniformization theorem for compact, genus-zero Riemann surface without boundary states that such surfaces can be conformally mapped to the Riemann sphere. 
Here, we use a generalization of this result for multiply-connected surfaces; 
see \cite[Theorem 17.1b]{Henrici1986}, \cite{Gardiner_1999}, \cite{Zeng_2009}, and \cite[p.123]{Jin_2018}. 

\begin{thm}[\cite{Gardiner_1999}]
Suppose $(\Sigma, g)$ is a 
smooth, compact, connected, genus-zero Riemann surface with $b$ boundary components. 
Then $\Sigma$ can be conformally mapped to a
unit disk with $b-1$ circular holes. 
That is, there exists a punctured unit disk with $b-1$ holes, 
$\Omega_{c,r} = D \ \setminus \  \cup_{i=1}^{b-1} D_i $, 
and a conformal map 
$f \colon (\Sigma,g) \to (\Omega_{c,r},\rho I)$, 
where $\rho I$ is a conformally flat metric. 
Furthermore, two such mappings differ by a M\"obius transformation.
\end{thm}

\begin{rem} \label{r:dim}
The uniqueness of the conformal map up to a M\"obius transformation means that it is possible to 
center one of the holes at the origin and center another hole on the positive $x$-axis. 
Thus, fixing these three parameters, the dimension of the parameter space of hole centers and radii
$\{c_i\}_{i=1}^{b-1} \cup  \{r_i\}_{i=1}^{b-1}$, 
is 1 for $b=2$ and $3b-6$ for $b\geq 3$, which is the dimension of the conformal module. 
\end{rem}

We now sketch a brief derivation of  \eqref{e:Steklov2} from \eqref{e:Steklov}. 
Let $f \colon (\Sigma,g) \to (\Omega_{c,r}, \rho I) $ be a conformal mapping.  
It is well-known that $v = u \circ f \colon \Sigma \to \mathbb R$ is harmonic if and only if  $ u \colon \Omega_{c,r} \to \mathbb R$ is harmonic \cite{olver2017complex}. This justifies  \eqref{e:Steklov2a}. 
We show \eqref{e:Steklov2b} on a flat domain for simplicity. Write $x = f(z)$ and  $v(z) = u \left( f(z) \right) = u(x)$, so that 
$\nabla_z v(z) = Df(z)^T \ \nabla_x u \left( f (z) \right)$. 
Since $Df(z) \ \nu(z) = | Df (z) | \ n  \left( f (z) \right)$, we have that 
\begin{align*}
\sigma u  \left( f (z) \right) &= \sigma v(z) \\
&= \nu^T(z) \ \nabla_z v(z) \\
&= \nu^T (z)  \ Df(z)^T \  \nabla_x v \left( f (z) \right) \\
&=  | Df (z) | \  n^T  \left( f (z) \right) \  \nabla_x u \left( f (z) \right) \\
&=  | Df (z) | \  \partial_n u \left( f (z) \right) 
\end{align*}
So, we obtain $\partial_n u (x) = \sigma \rho(x) u(x)$, where $\rho(x) = | Df \left( f^{-1}(x) \right) |^{-1} = |D h (x) |$, where $h = f^{-1}$. 

\bigskip

Remark~\ref{r:dim} shows that our parameterization of  $\Omega_{c,r}$ is over-complete, as the following example further demonstrates. 

\begin{ex} \label{e:ex2}
Denote $\Omega_{1}$ as an eccentric annulus with boundaries
\[
\left|z-c_{1}\right|<r_{1}\quad\text{and}\quad\left|z\right|<1
\]
and $\Omega_{2}$ as an concentric annulus $r_{2}<|x|<1$ where $c_{1}$,
$r_{1}$, $r_{2}$ are real numbers and $x,z\in\mathbb{C}$. 
A conformal mapping $h:\Omega_{1}\rightarrow\Omega_{2}$ is given by
\[x=f(z)=\frac{z-a}{1-az}\]
where $a$ and $r_{2}$ are determined by mapping $c_1+r_1, c_1-r_1$ to $r_2,-r_2$ and satisfy 
\begin{align*}
a=\frac{1+c_{1}^{2}-r_{1}^{2}-\sqrt{\left(1+c_{1}^{2}-r_{1}^{2}\right)^{2}-4c_{1}^{2}}}{2c_{1}}, 
\quad \textrm{and} \quad
r_{2}  =\frac{r_1+c_1-a}{1-a(r_1+c_1)}.
\end{align*}
In this example, $z=h(x)=\frac{x+a}{1+ax}$,
and 
\[
\rho(x)=|z_{x}|=\left|\frac{1-a^{2}}{\left(1+ax\right)^{2}}\right|.
\]
In Figure~\ref{f:conformal_mapping}, the mapping is shown for $c_{1}=r_{1}=\frac{1}{4}$ and
the resulting $a=r_{2}=2-\sqrt{3}$. 

Thus, the eccentric annulus $\Omega_1$ with boundary density $\rho = 1$ has the same Steklov spectrum as the concentric annulus $\Omega_2$ with boundary density $\rho(x)$ given above. 
In particular, this example shows that the decomposition of perturbations of a metric  into conformal and non-conformal directions is not equivalent to either changing $(c_1,r_1)$ or $\rho$, respectively. 
While changing $\rho$  is a conformal perturbation, a change in $(c_1,r_1)$ gives a perturbation to the metric that has components in both the conformal and non-conformal directions. 
\end{ex}

\begin{figure}
    \centering
    \includegraphics[height=0.3\textwidth]{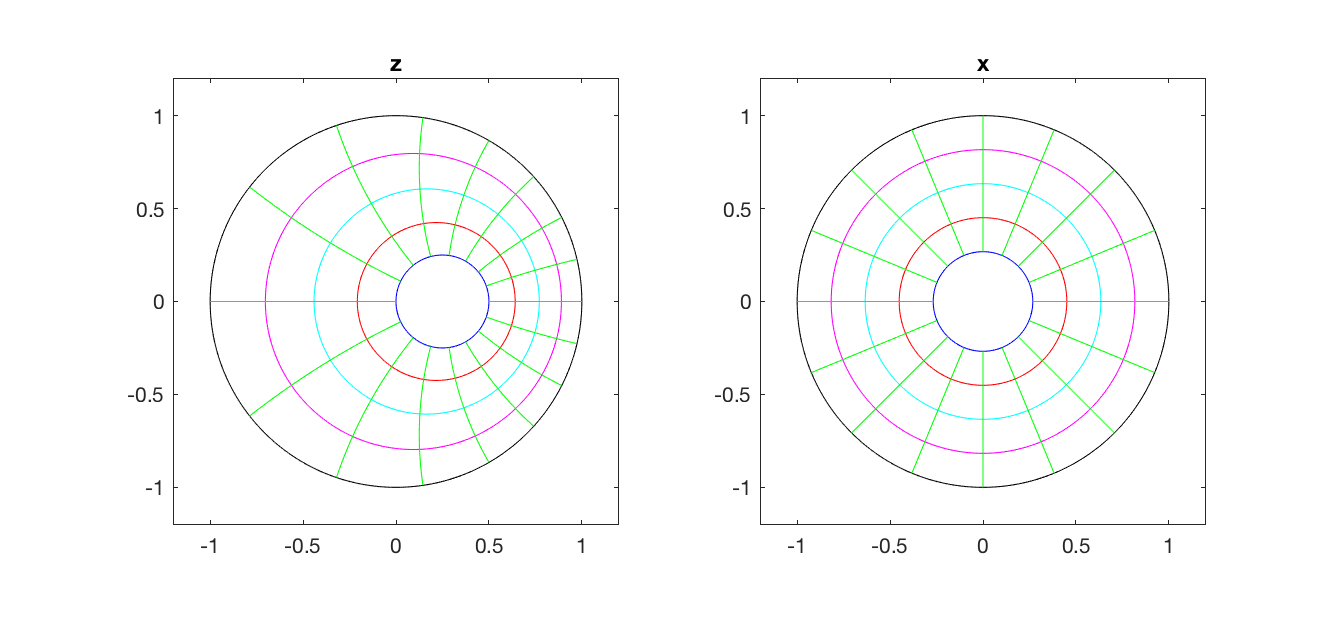}
    \caption{A conformal mapping from an eccentric annulus to a concentric annulus. See Example~\ref{e:ex2}.}
    \label{f:conformal_mapping}
\end{figure}

The following two examples illustrate what happens to the boundary density $\rho$ when $\Sigma$ becomes ``pinched''. 

\begin{ex} \label{e:Hippopede}
We consider the conformal mapping $h\colon D \to \Omega_\alpha$ from the unit
disk $|x|\le1$ to the Hippopede domain, $\Omega_\alpha$,
\[
h(x)=\frac{2\alpha x}{1+\alpha+(1-\alpha)x^{2}};
\]
see  \cite{girouard2016steklov, alhejaili2019maximal}.  When $\alpha=1$, this is the identity mapping on the unit disk and  as 
$\alpha \rightarrow 0^{+}$, 
it maps a unit disk to two ``kissing'' disks. 
In the left and center panels of Figure \ref{f:Hippopede}, the mapping is shown for $\alpha=\frac{1}{10}$. 
Here, we compute, 
\[
\rho_\alpha(x)=\left|\frac{2\alpha\left(1+\alpha-(1-\alpha)x^{2}\right)}{\left(1+\alpha+(1-\alpha)x^{2}\right)^{2}}\right|.
\]
Let $x=e^{i\theta}$. In the right panel of Figure \ref{f:Hippopede}, 
we plot $\rho_\alpha (\theta)$ for $\alpha = \frac{1}{50}$, $ \frac{1}{10}$, and $\frac{1}{5}$.
We observe that $\rho_\alpha(\theta)$ becomes singular as $\alpha\rightarrow0^{+}$
at $\theta=\frac{\pi}{2}$ and $\frac{3\pi}{2}$. 
\end{ex}

\begin{figure}
    \centering
    \includegraphics[height=0.3\textwidth]{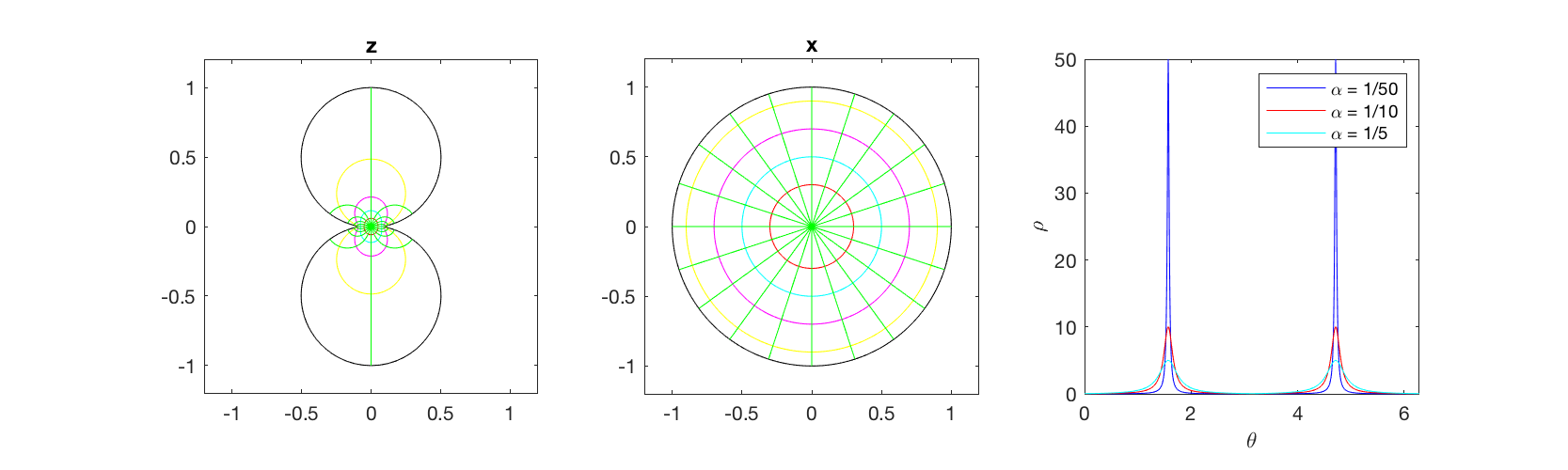}
    \caption{A conformal mapping from a Hippopede shape to a unit disk. See Example~\ref{e:Hippopede}.}
    \label{f:Hippopede}
\end{figure}

In Example~\ref{e:Hippopede}, the density is singular at two points. The following example illustrates how the density function can become singular at a single point. 

\begin{ex} \label{e:ex3}
We consider a radius $r_{1} :=0.833$ disk, $\Omega_{1}=\{|x|< r_1\}$, (see Figure~\ref{f:critical_caternoid_glue_disk_to_a_disk}(c)) 
and a domain $\Omega_{2}$ consisting of the union of two disks with radii $r_1$ and $1$ and
a `neck' of width $2\alpha$  (see Figure~\ref{f:critical_caternoid_glue_disk_to_a_disk}(a)). 
Define the conformal mapping
$h:\Omega_{1}\rightarrow\Omega_{2}$  
as the composition of the two functions $h=h_{1}\circ h_{2}$,  where 
\begin{align*} 
z := h_{1}(y)  =\frac{y-ic}{ay^{2}+b}+i\alpha c, 
\qquad \textrm{and} \qquad 
y: = h_{2}(x)  =\frac{\frac{x}{r_{1}}-i(1-\beta)}{1+i(1-\beta)\frac{x}{r_{1}}}.
\end{align*}
The constants $a,b,c$ are chosen as 
\[
a=\frac{1}{2}(\frac{1}{\alpha}-\frac{1}{r_{1}+1}), \qquad 
b=\frac{1}{\alpha}-a, \qquad 
c=\frac{\alpha+4a\alpha-2}{2a\alpha^{2},}
\]
so that $h_{1}$ maps $1,i,-1,-i$ to $\alpha,2i,-\alpha,-2r_{1}i$, respectively. See Figure~\ref{f:critical_caternoid_glue_disk_to_a_disk}(a) and (b). 
The constant $\beta$ is chosen so that $\Omega_{1}$ maps to a unit
disk and the zero in $\Omega_{1}$ maps to $-i(1-\beta).$ See Figure~\ref{f:critical_caternoid_glue_disk_to_a_disk}(b) and (c).When $\beta$
is small, this function maps points which are uniformly distributed
on $\partial\Omega_{1}$ to points that  accumulate near $-i$ on the unit disk. The boundary density, $\rho$, 
can be obtained via the product rule, 
\[
\rho(x)=|h_{x}|=|h_{1}^{'}(h_{2}(x))h_{2}^{'}(x)|,\quad\text{for}\quad|x|=r_{1}.
\]
As shown in Figures~\ref{f:critical_caternoid_glue_disk_to_a_disk}(d), the density reaches a large value at $\theta=\frac{\pi}{2}$. 
Figure \ref{f:critical_caternoid_glue_disk_to_a_disk}(e) shows the detail profile of the density function about one. 
\end{ex}

\begin{figure}
    \centering
    \includegraphics[height=0.4\textwidth]{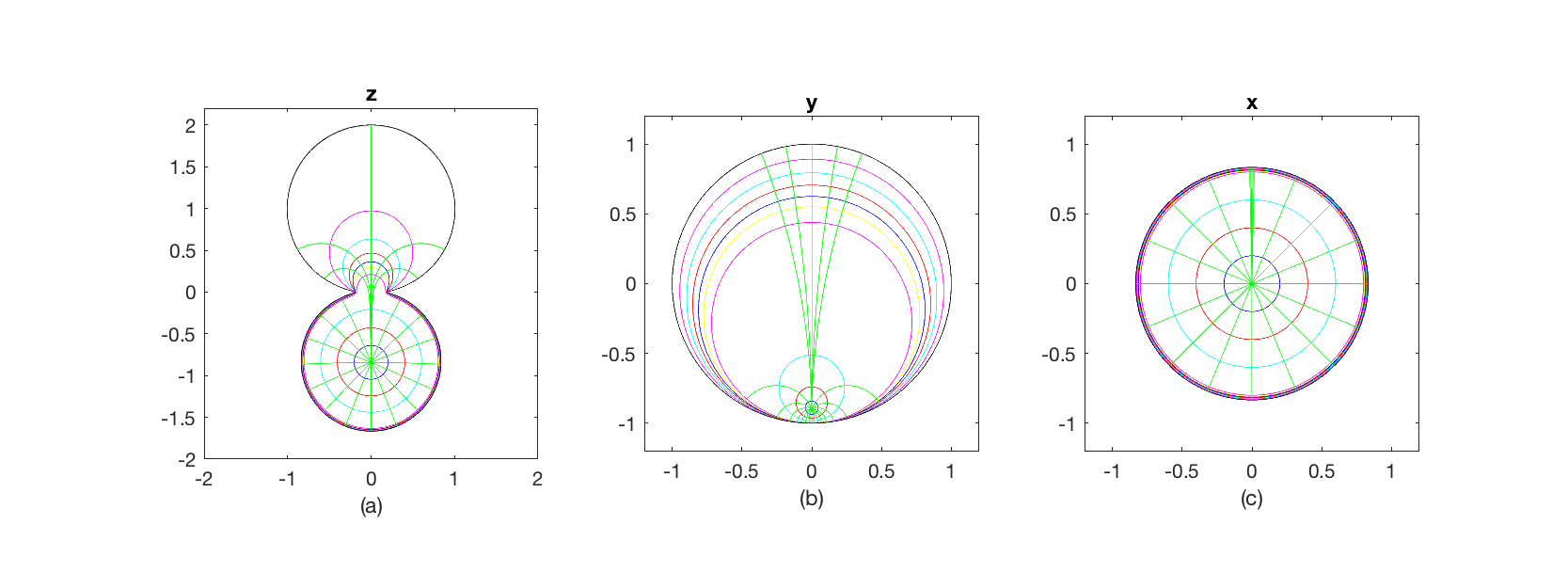}
    \includegraphics[height=0.3\textwidth]{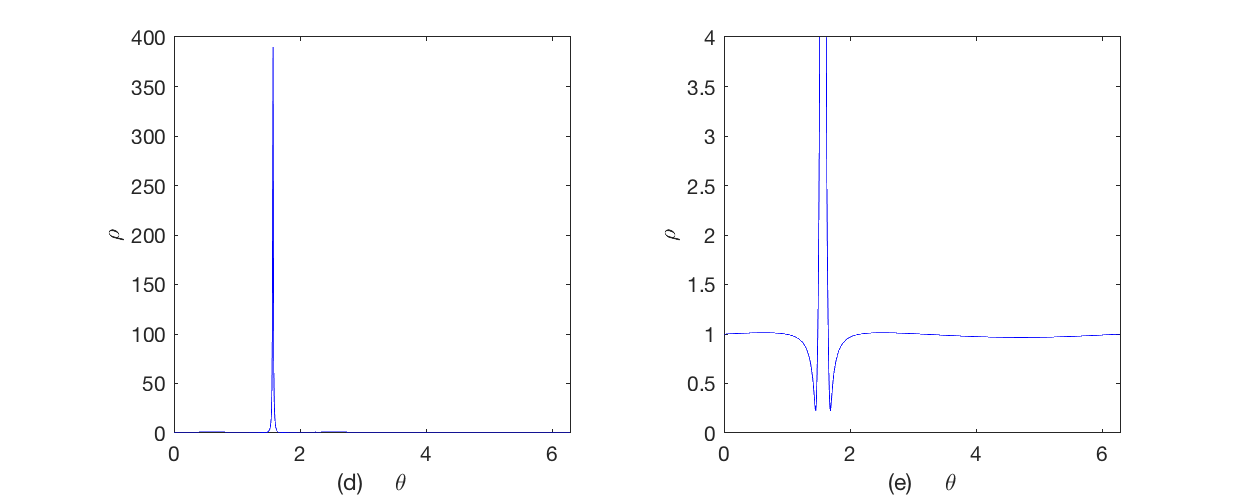}
    \caption{A conformal mapping from a disk to a shape which is close to the union of two disks. The choice of parameters are $\alpha=0.2$ and $\beta=0.1$. See Example~\ref{e:ex3}. }
    \label{f:critical_caternoid_glue_disk_to_a_disk}
\end{figure}

\subsection{Eigenvalue derivatives with respect to the density and shape parameters}
\label{s:EigDeriv}

In this section, we consider $\sigma$ and $ \tilde \sigma = \sigma L$ as a function of $\rho$ and the shape $\Omega_{c,r}$.
We first compute the derivatives with respect to $\rho$. 
\begin{prop} 
Let $(\sigma, u)$ be a  simple Steklov eigenpair, satisfying \eqref{e:Steklov2}, normalized so that $\int_{\partial \Omega_{c,r}} \rho u^2 = 1$.
Then the functionals $\rho \mapsto \sigma$ and $\rho \mapsto \tilde \sigma$ are Frech\'et differentiable with derivatives 
\begin{subequations}
\label{e:Drho}
\begin{align}
\langle \frac{\delta \sigma}{\delta \rho} , \delta \rho \rangle &= - \sigma \int_{\partial \Omega_{c,r}} u^2(x)  \delta \rho(x)  \ dx, \\
\langle \frac{\delta \tilde \sigma }{\delta \rho} , \delta \rho \rangle &= \sigma \int_{\partial \Omega_{c,r}} \left(1 - L u^2(x)\right)  \delta \rho(x)  \ dx . 
\end{align}
\end{subequations}
\end{prop}
\begin{proof}
We take variations of the formula $\sigma = \int_{\Omega_{c,r}} | \nabla u |^2 \ dx$ and use Green's identity to obtain 
\begin{align*}
\dot \sigma &= 2 \int_{\Omega_{c,r}} \nabla u \cdot \nabla \dot u \ dx \\
&= -  2 \int_{\Omega_{c,r}}  \dot u \Delta u \ dx + 2 \int_{\partial \Omega_{c,r}} \dot u u_n  \ dx \\
&= 2\sigma  \int_{\partial \Omega_{c,r}} \rho u \dot u  \ dx. 
\end{align*}
From the normalization condition, $\int_{\partial \Omega_{c,r}} \rho u^2 \ dx = 1$, we obtain 
$$
\int_{\partial \Omega_{c,r}} \dot \rho u^2 \ dx = - 2 \int_{\partial \Omega_{c,r}}  \rho u \dot u \ dx,
$$
which gives the desired result. The derivative of $\tilde \sigma$ is obtained via $L = \int_{\partial \Omega_{c,r}} \rho \ dx$ and the product rule. 
\end{proof}

We describe below optimality conditions when the multiplicity of the optimized eigenvalue 
is greater than one. Our formulation is highly inspired by previous articles 
\cite{ElSoufi_2007,Fraser_2015,Bogosel_2016}. We first need the following 
regularity result; see \cite[Theorem 3.2]{Lamberti_2015}. 

\begin{lem} \label{l:branches}
    Let $\sigma(\rho)$ be an eigenvalue of multiplicity $p>1$ of
     system \eqref{e:Steklov2} associated to a smooth domain $\Omega$ with 
     nonnegative boundary density $\rho$. Let 
     $\delta \rho \in L^2(\partial \Omega)$ and consider the eigenvalues 
     associated to the densities $\rho_\varepsilon = \rho + \varepsilon \delta \rho$ for $\varepsilon 
     \in \mathbb R$.  There exists 
     $\varepsilon_0> 0$ and nontrivial  functions
      $(\sigma_i(\varepsilon))_{1 \leq i \leq p}$ and  
      $(u_i(\varepsilon))_{1 \leq i \leq p}$  analytic on $(- \varepsilon_0, \varepsilon_0)$
      such that for all $i = 1,\dots, p$:
      \begin{itemize}
        \item[(a)] $\sigma_i(0) = \sigma(\rho)$,
        \item[(b)] The family $\{u_1(\varepsilon), \dots, u_p(\varepsilon) \}$ is 
        orthonormal in $L^2(\partial \Omega, \rho_\varepsilon )$, 
        \item[(c)] Every couple $(\sigma_i(\varepsilon), u_i(\varepsilon))$ is
        solution of system  \eqref{e:Steklov2} for the density $\rho_\varepsilon$.
      \end{itemize}
\end{lem}
We can now evaluate directional derivatives based on previous parametrizations:
\begin{lem} \label{l:quadform}
    Let $\sigma$ be an eigenvalue of multiplicity $p> 1$ of the weighted Steklov 
    system \eqref{e:Steklov2} for some nonnegative boundary density $\rho$. Denote 
    by $E_\sigma$ the corresponding eigenspace. Let 
    $\rho_\varepsilon = \rho + \varepsilon \delta \rho$ be a perturbation of $\rho$ for 
    some $\delta \rho \in L^2(\partial \Omega)$ .
     Let $(\sigma_i(\varepsilon))_{1 \leq i \leq p}$ and
     $( u_i(\varepsilon))_{1 \leq i\leq p}$ be some smooth parametrizations as the ones 
     given by Lemma \ref{l:branches}. Then $\sigma_i' = \frac{d}{d\varepsilon} \sigma_{i}(\varepsilon)|_{\varepsilon = 0}$ 
     are the eigenvalues of the quadratic form $q_{\delta \rho}$ defined on $E_\sigma \subset L^2(\partial \Omega, \rho)$ by
  
  \[ q_{\delta \rho}( u) = - \sigma \int_{\partial \Omega}   u^2 \delta \rho  \ d x.\]
  Moreover, the $L^2(\partial \Omega, \rho)$-orthonormal basis $ u_1(0),..., u_p(0)$ diagonalizes $q_{\delta \rho}$ on $E_\sigma$.

\end{lem}

\begin{proof} Let   $(\sigma_i(\varepsilon))_{1 \leq i \leq p}$ and  
    $(u_i(\varepsilon))_{1 \leq i \leq p}$  defined on $(- \varepsilon_0, \varepsilon_0)$
    for some $\varepsilon_0 >0$ satisfying properties of Lemma \ref{l:branches}.
    For all  $\varepsilon \in (- \varepsilon_0, \varepsilon_0)$, $i = 1,\dots, p$ and 
    $ v \in  L^2(\partial \Omega, \rho)$,  we have from \eqref{e:Steklov2}, that 
    \begin{equation} \label{eqvarv}
        \int_\Omega \nabla u_i(\varepsilon) \cdot \nabla v \ dx = \sigma_i(\varepsilon)
        \int_{\partial \Omega} u_i(\varepsilon) v \rho_\varepsilon \ dx.
    \end{equation} 
    Differentiationg this equality with respect to $\varepsilon$ and evaluating at
    $\varepsilon = 0$ gives
    $$ \int_\Omega \nabla u_i'(0) \cdot \nabla v \ dx = 
    \sigma \int_{\partial \Omega} u_i(0) v \delta \rho \ dx +
    \sigma \int_{\partial \Omega} u_i'(0) v \rho \ dx +
    \sigma_i' \int_{\partial \Omega} u_i(0) v \rho \ dx.
    $$
    Thus, with $v=u_j(0)$ and using  \eqref{eqvarv} replacing $i$ per $j$ and  $v$ 
    by   $u'_i(0)$, we obtain 
    $$   \sigma_i' \int_{\partial \Omega} u_i(0) u_j(0) \rho \ dx = - \sigma  \int_{\partial \Omega}   u_i(0) u_j(0) \delta \rho  \ d x$$
    which exactly  proves that  $L^2(\partial \Omega, \rho)$-orthonormal basis $
     u_1(0),..., u_p(0)$ diagonalizes $q_{\delta \rho}$ on $E_\sigma$. Moreover, 
      the $ \sigma_i'$ are eigenvalues of this quadratic form.
\end{proof}

We can now establish optimality conditions with respect to the boundary density
in case of multiple eigenvalues. 
\begin{prop} \label{p:UnitSphere} Let $j \geq 1$ and $\Omega$ a smooth domain of $\mathbb{R}^2$. Assume a nonnegative 
    $\rho \in L^2(\partial \Omega)$ maximizes the product $\sigma_j(\rho) L(\rho)$
     among all nonnegative functions of $L^2(\partial \Omega)$ where $L(\rho) = \int_{\partial \Omega} \rho \ dx$
  and $\sigma_j(\rho)$ is the $j$-th eigenvalues of system \eqref{e:Steklov2}.
If $\sigma_j(\rho)$ is of multiplicity $p > 1$ and $E_{\sigma_j}$ its eigenspace, there 
exists a basis of $p$ functions $u_1, \dots ,u_p$ of $E_{\sigma_j}$ which satisfy 
$$
\sum_{i=1}^p u_i(x)^2 = 1
$$
for all $x \in \partial \Omega$.
\end{prop}

\begin{proof} The proposition is an almost direct consequence of Lemma \ref{l:quadform}
    and of Hahn-Banach separation theorem. Consider the convex hull
    $K=\text{Co}\left\{ u^2,\  u \in E_{\sigma_j} \right\}$. 
    We want to prove that the function identically 
    equal to one belongs to $K$. If it is not the case, by Hahn-Banach theorem applied to the finite 
    dimensional normed vector subspace of $C^1(\partial \Omega)$ spanned by $K$ and $1$, 
    there exists a function 
    $\delta \rho \in C^1(\partial \Omega)$ such that $\int_{\partial \Omega}  \delta \rho \ dx >0$
     and which satisfies, for all $ u \in E_{\sigma_j}$,
    \[ \int_{\partial \Omega} u^2 \delta \rho \ d x \leq 0.\]
    This last inequality asserts that the quadratic form $q_{\delta \rho}$ on 
    $E_{\sigma_j}$ has nonnegative eigenvalues. Thus, both the $p$ eigenvalues and 
     the weighted length increase in the direction of $\delta \rho$. 
     As a consequence, for $\varepsilon$ small enough, the 
    product of $\sigma_j(\rho + \varepsilon \delta \rho) L(\rho + \varepsilon \delta \rho)$
     is strictly greater than $\sigma_j(\rho) L(\rho)$ due to the strict inequality 
     of the separation result which contradicts the optimality.
\end{proof}

To compute the derivatives of $\sigma$ and $\tilde \sigma$ with respect to the centers $c$ and radii $r$, 
we first compute the shape derivative with respect to perturbations of the boundary of $\Omega_{c,r}$. This result extends a result in \cite{dambrine2014extremal,Akhmetgaliyev_2017,Bogosel_2017} to $\rho \neq 1$. 
\begin{prop}  \label{p:ShapeDer}
Consider the perturbation $x\mapsto x + \tau v$.  Then a simple (unit-normalized) Steklov eigenpair $ (\sigma, u)$ satisfies the perturbation formula
\begin{equation}
\sigma^{'} = \int_{\partial \Omega}  \left(  |\nabla u |^2  - 2 \rho^2 \sigma^2 u^2 - \sigma \kappa \rho u^2 \right)  (v \cdot \hat n)  +  \sigma \rho_t  u^2 (v \cdot \hat t) \ dx, 
\end{equation}
where $\hat n$ is the outward unit normal vector, 
$\hat t$ denotes the tangential direction, and
where $\kappa$ is the signed curvature of the boundary. 
We also have $L' = \int_{\partial \Omega}  \kappa \rho  (v \cdot \hat n)  - \rho_t  (v \cdot \hat t) \ dx $. 
\end{prop}
\begin{proof}
We follow the proof in \cite{Akhmetgaliyev_2017}.  Let primes denote the shape derivative. 
From the identity $\sigma = \int_\Omega |\nabla u |^2 \ dx$, we compute
\begin{subequations}
\label{e:dsigma}
\begin{align}
\sigma' &= 2\int_\Omega \nabla u \cdot \nabla u' \ dx + \int_{\partial \Omega}  |\nabla u |^2  (v\cdot \hat n) \ dx 
&& \textrm{(shape derivative)} \\
&= - 2 \int_{\Omega} (\Delta u) u'  \ dx + 2 \int_{\partial \Omega} u_n u'  \ dx + \int_{\partial \Omega}  |\nabla u |^2  (v\cdot \hat n) \ dx 
&& \textrm{(Green's identity)} \\
&= 2\sigma \int_{\partial \Omega} \rho u u'  \ dx + \int_{\partial \Omega}  |\nabla u |^2  (v\cdot \hat n) \ dx
&& \textrm{(Equation \eqref{e:Steklov2})}. 
\end{align}
\end{subequations}
Differentiating  the normalization equation, $ \int_{\partial \Omega} \rho u^2  \ dx= 1$, we have that 
\begin{align*}
 2 \int_{\partial \Omega} \rho u u' \ dx = -  \int_{\partial \Omega}  \rho' u^2 +  \left(  \partial_n (\rho u^2 ) + \kappa  \rho u^2\right) (v\cdot \hat n) \ dx, 
 \end{align*}
 where $\kappa$ is the curvature of the boundary and $\rho' = - \nabla \rho \cdot v$.  
 Extending $\rho$ constantly in the normal direction, we  have $\rho' + (v\cdot \hat n) \rho_n = -\rho_t (v\cdot t)$ where $t$ denotes the tangential direction. 
 We then have that 
\begin{align*}
 2 \int_{\partial \Omega} \rho u u' \ dx =   \int_{\partial \Omega}  \rho_t  u^2 (v \cdot t) -  \left(  2 \rho u u_n + \kappa  \rho u^2\right) (v\cdot \hat n) \ dx. 
 \end{align*} 
Combining this with \eqref{e:dsigma}, we obtain the desired result. 
 \end{proof}
 
 Using Proposition~\ref{p:ShapeDer}, we can now compute the derivatives of $\sigma$ and $\tilde \sigma$ for the domain 
 $ \Omega_{c,r} = D \ \setminus \  \cup_{i=1}^{b-1} D_i $ 
with respect to a center $c_i$ and radius $r_i$ of $D_i$ as follows. 
To compute the derivative with respect to $r_i$, we choose a perturbation $v$ so that 
 $$
 v \cdot \hat n = -1 
 \qquad \textrm{and} \qquad
 v \cdot \hat t = 0 
 \qquad \textrm{on} \ \partial D_i. 
 $$
 Then, noting that $\kappa =  -1/r_i$, we obtain
\begin{equation}
\label{e:Dr}
\frac{\partial \sigma}{\partial r_i} = - \int_{\partial D_i}   |\nabla u |^2  - 2 \rho^2 \sigma^2 u^2 +  \frac{\sigma}{r_i} \rho u^2     \ dx. 
\end{equation}
To compute the derivative with respect to $c_i$, we take two perturbations $v$ of the form  
 $$
 v \cdot \hat n = \cos \theta 
 \qquad \textrm{and} \qquad
 v \cdot \hat t =  \sin \theta
 \qquad \textrm{on} \ \partial D_i 
 $$
 and 
 $$
  v \cdot \hat n = \sin \theta 
 \qquad \textrm{and} \qquad
 v \cdot \hat t =  -\cos \theta
 \qquad \textrm{on} \ \partial D_i, 
$$
 to obtain 
\begin{equation}
\label{e:Dc}
\nabla_{c_i} \sigma = \int_{\partial \Omega}  \left(  |\nabla u |^2  - 2 \rho^2 \sigma^2 u^2 + \frac{\sigma}{r_i}  \rho u^2 \right) 
\begin{pmatrix} \cos \theta \\ \sin \theta \end{pmatrix}  
+  \sigma \rho_t  u^2
\begin{pmatrix} \sin \theta \\ 
-\cos \theta \end{pmatrix}  
 \ dx. 
\end{equation}

\begin{rem} \label{r:IsoCoords}
In \cite{Fraser_2015}, a detailed study of perturbations to the metric yield two conditions for a maximal Steklov eigenvalue. 
The first comes from the study of perturbations in  ``conformal directions'' and, as in Proposition~\ref{p:UnitSphere}, result in the existence of eigenfunctions $\{ u_j\}_{j=1}^n$ such that the map 
$U = [u_1 | \cdots | u_n] \colon \Omega \to \mathbb B^n$ 
satisfies 
$U(\partial \Omega) \subset \mathbb S^{n-1}$. 
The second condition comes from the study of non-conformal perturbations of the metric and give that the map 
$U\colon \Omega \to \mathbb B^n$ 
has isothermal coordinates, \ie, satisfies 
\begin{align*}
& | \partial_x U | = | \partial_y U |, \\
& \partial_x U \cdot \partial_y U = 0. 
\end{align*}
Since a change in the parameters $(c,r)$ gives a perturbation to the metric that has components in both the conformal and non-conformal directions (see Remark~\ref{r:dim} and  Example~\ref{e:ex2}), this second condition is nontrivial to obtain from \eqref{e:Dr} and \eqref{e:Dc}. 
\end{rem}

\section{Steklov eigenvalues of rotationally symmetric annuli and the critical catenoid}
\label{s:CC}

Here, we discuss the Steklov eigenvalues of rotationally symmetric annuli,  the critical catenoid, and coverings of the critical catenoid. 
These results are also discussed in  \cite{Fraser_2011,Fan_2014}  using cylindrical coordinates, but it useful to review these computations and have them written in annular coordinates for comparison and discussion; see also \cite{Martel_2014,dittmar2004sums}. 

\subsection{Steklov eigenvalues of rotationally symmetric annuli} \label{s:SymAnn}
Here,  for $s\in (0,1)$, we consider the rotationally symmetric annulus, 
$$
A_s = \{ (r,\theta) \colon r \in [s, 1] \},
$$ 
and explicitly compute Steklov eigenvalues satisfying 
\begin{subequations} 
\label{e:cenAnn}
\begin{align}
& [r^{-1} \partial_r r \partial_r + r^{-2} \partial_\theta^2] u  = 0 && (r,\theta) \in A_s, \\
& \partial_\nu u = \sigma \rho_s u && r = s, \\
& \partial_\nu u = \sigma \rho_1 u && r = 1. 
\end{align}
\end{subequations}
Note that if 
$(\sigma,u)$ is an eigenpair satisfying \eqref{e:cenAnn} with parameters $(s,\rho_s, \rho_1)$, then for $\alpha > 0$, 
$(\sigma/\alpha,u)$ is an eigenpair satisfying \eqref{e:cenAnn} with parameters $(s, \alpha \rho_s,  \alpha \rho_1)$. 
Using separation of variables, we obtain general solutions to the Laplace equation of the form 
$$
u(r,\theta) = C_1 + C_2 \log(r) + \sum_{k=1}^{\infty} ( C_3 r^k + C_4 r^{-k} ) ( C_5 \cos k \theta  +  C_6 \sin k \theta), 
$$
where $C_1, \ldots, C_6$ are constants. Using the Steklov boundary conditions, we can determine the eigenpairs, $(\sigma, u)$. Of course, there is a trivial eigenvalue, $\sigma_0 = 0$ with corresponding constant eigenfunction.
There is another eigenpair with eigenfunction that is constant in $\theta$, given by 
$$
\sigma =  \frac{\rho_1  + s \rho_s}{ \rho_1 \rho_s s } \frac{1}{\log s^{-1} }, 
\qquad \qquad  
u(r,\theta) = 1 + \sigma \rho_1 \log r. 
$$
We note that $L = 2\pi (\rho_1 + s \rho_s)$, so that 
$$
\tilde \sigma = \sigma L =  2 \pi  \frac{ (\rho_1  + s \rho_s)^2 }{ \rho_1 \rho_s s } \frac{1}{\log s^{-1} }. 
$$
For each $k = 1, 2, \ldots$, there are also eigenfunctions that are oscillatory in $\theta$ of the form 
$$
u(r,\theta) = (A r^k + B r^{-k} ) \{\cos k \theta,  \ \sin k \theta \}, 
$$
where $A$, $B$ are constants. Here, the brackets indicate that we can choose either $\cos$ or $\sin$; the corresponding eigenvalue has multiplicity two. 
Using the boundary conditions we obtain the $2\times 2$ generalized eigenproblem, 
$$
\begin{pmatrix}
k & -k \\
- k s^{k-1} & k s^{-k-1} 
\end{pmatrix}
\begin{pmatrix}
A \\ B 
\end{pmatrix}
= \sigma
\begin{pmatrix}
\rho_1 & \rho_1 \\
\rho_s s^{k} & \rho_s s^{-k} 
\end{pmatrix}
\begin{pmatrix}
A \\ B 
\end{pmatrix}.
$$
This is equivalent to the eigenproblem
$$
\frac{k}{ \rho_1 s \rho_s \sinh(-k \log s) }  
\begin{pmatrix}
s \rho_s s^{-k}  + \rho_1 s^k & - s \rho_s s^{-k} - \rho_1 s^{-k} \\
- s \rho_s s^k - \rho_1 s^k  & s \rho_s s^k + \rho_1 s^{-k} 
\end{pmatrix}
\begin{pmatrix}
A \\ B 
\end{pmatrix}
= \sigma
\begin{pmatrix}
A \\ B 
\end{pmatrix},
$$
from which one obtains the real positive eigenvalues 
$$\sigma_{k,\pm} = 
 \frac{k }{2 \rho_1 s \rho_s} \coth(-k \log s) \left[ \rho_1 + s \rho_s \pm \sqrt{ \left( \rho_1 + s \rho_s\right)^2 - 4 \rho_1 s \rho_s \tanh^2(-k \log s) } \right]. 
$$

In Figure~\ref{f:StekEigCenAnn}(left), for 
$\rho_s/\rho_1 = 11.01609$, 
we display the length-normalized Steklov eigenvalues for various values of $s$.  
The eigenvalue corresponding to the radially symmetric eigenfunction is plotted in red. 
The thin vertical line indicates the value $s = 0.090776$. 
For this value of $s$, the first Steklov eigenvalue has multiplicity three and  length-normalized eigenvalue $\tilde \sigma = 10.47478$.  
In Figure~\ref{f:StekEigCenAnn}(right), we plot contours of two of the eigenfunctions; the third can be obtained by rotating the image of the lower eigenfunction by $\frac{\pi}{2}$.

\begin{figure}[t]
\centering
\includegraphics[height=0.5\textwidth]{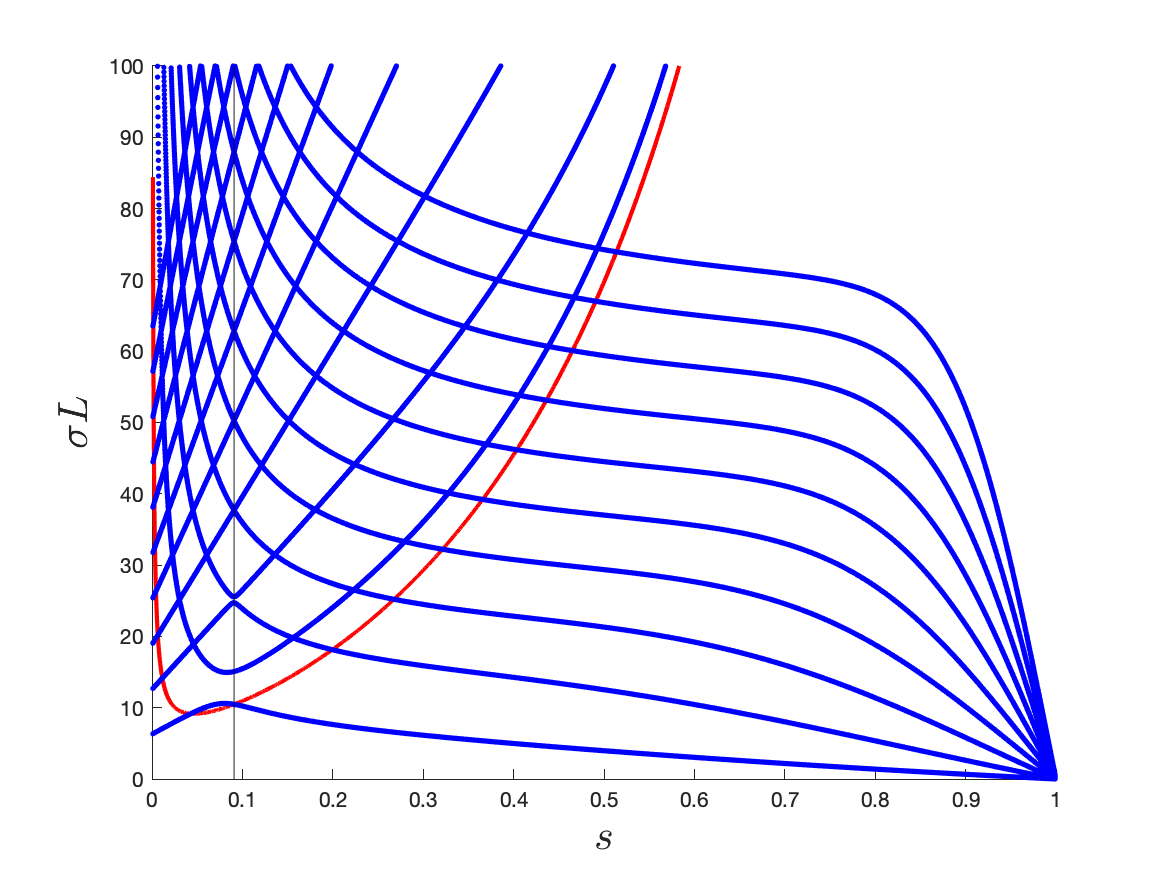}
\includegraphics[trim=25 35 20 25, clip,height=0.5\textwidth]{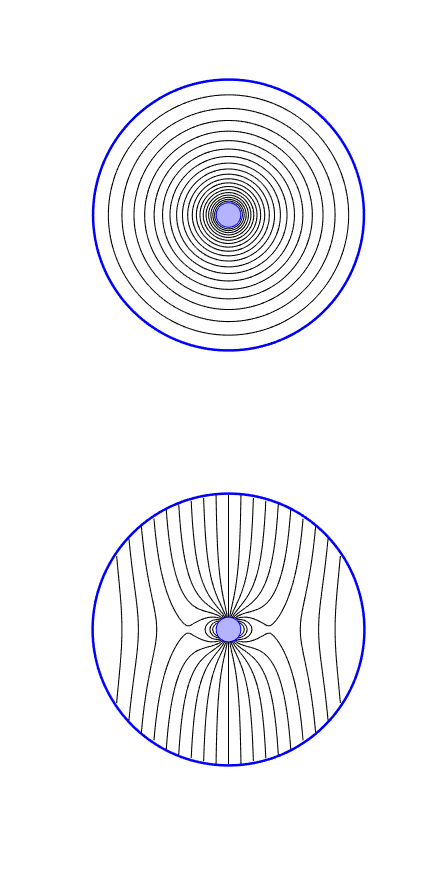}
\caption{{\bf (left)} Length normalized Steklov eigenvalues of the annulus, $A_s$ for varying inner radius $s$.  The blue lines represent multiplicity two eigenvalues for different values of $k$, while the red line represents a multiplicity one eigenvalue.  
{\bf (right)} For $s = 0.090776$, we plot contours of eigenfunctions corresponding to the first Steklov eigenvalue. See Section~\ref{s:SymAnn}.}
\label{f:StekEigCenAnn}
\end{figure}

\subsection{Extremal eigenvalues for rotationally symmetric annuli} 
We consider the extremal eigenvalue problem for rotationally symmetric annuli, 
\begin{align}
\label{e:ExtStekCenAnn}
\max_{s, \rho_s, \rho_1}  \ \tilde \sigma_j, 
\qquad \qquad 
\tilde \sigma_j := \sigma_j L.
\end{align}
Here, $\sigma_j$ is assumed to satisfy \eqref{e:cenAnn}. 

\subsubsection{The first eigenvalue} 
We first consider $j=1$. By the symmetry of $\rho_1$ and $s \rho_s$, we obtain the optimality condition 
$$
s\rho_s = \rho_1 =: \rho.
$$
In this case, we have the two length-normalized eigenvalues and associated $L^2(\partial \Omega, \rho)$-normalized eigenfunctions
\begin{align*}
\sigma_{1,-} L &= 4 \pi \frac{1-s}{1+s}, && 
u(r,\theta) =\frac{1}{\sqrt{2 \pi \rho} }  \frac{ \cosh\left( \log \frac{r}{\sqrt s} \right) }{ \cosh\left( \log \sqrt s \right) } \{ \cos \theta,  \ \sin \theta \} \\
\sigma L &= \frac{8 \pi}{ \log s^{-1}},  && 
u(r,\theta) =  \frac{1}{\sqrt{4 \pi \rho } } \frac{\log \frac{r}{\sqrt{s}}}{ \log \sqrt{s}} . 
\end{align*}
The two values of $\sigma L$ are equal when $s$ is the unique solution of the transcendental equation 
$$
\frac{1+s}{ 1-s } =-  \log \sqrt{s}, 
\qquad \qquad s > 0. 
$$
The solution is approximately given by $s = 0.090776$.

We now consider the map $U \colon A_s \to \mathbb B^3$, defined by 
$$
U(r,\theta) = 
\begin{pmatrix}
\frac{\cosh\left( \log \frac{r}{\sqrt s} \right) }{\sqrt{ \cosh^2(\log \sqrt s)  +  \log^2\sqrt{s} }}  \cos \theta \\
\frac{ \cosh\left( \log \frac{r}{\sqrt s} \right) }{\sqrt{ \cosh^2(\log \sqrt s)  +  \log^2\sqrt{s} }} \sin \theta \\
 \frac{\log \frac{r}{\sqrt s}}{ \sqrt{ \cosh^2(\log \sqrt s) + \log^2(\sqrt s)}} 
\end{pmatrix}, 
\qquad \qquad 
 (r,\theta) \in A_s. 
$$
Note that this map has coordinates that are linear combinations of the above eigenfunctions. 
One can check that these are isothermal coordinates, \ie, 
\begin{align*}
& |\partial_r U(r,\theta) |^2 =  r^{-2} | \partial_\theta U(r,\theta) |^2, &&\forall (r,\theta) \in A_s, \\
& \partial_r U(r,\theta) \cdot r^{-1} \partial_\theta U(r,\theta)  = 0, &&\forall (r,\theta) \in A_s,
\end{align*}
and satisfy  $U(\partial A_s) \subset \mathbb{S}^2 \subset \mathbb R^3$, \ie, 
$$
|U(1,\theta)|^2 = |U(s,\theta)|^2 = 1,
\qquad \qquad 
\forall \theta \in [0,2\pi].
$$ 

Furthermore, it is not difficult to check that $U(A_s)$ is the critical catenoid. That is,  
$$
U(A_s) = C_{\alpha^*}
$$
where 
$$
C_\alpha = \left\{ x\in \mathbb R^3 \colon \sqrt{x_1^2 + x_2^2} = \alpha \cosh \left( \frac{x_3}{\alpha} \right) \right\}, 
\qquad \qquad
\alpha > 0, 
$$  
is a \emph{catenoid} and the \emph{critical catenoid} is the catenoid  with 
 $\alpha = \alpha^* = \left( \beta^2 + \cosh^2\beta \right)^{- \frac 1 2 }$ where $\beta =  - \log \sqrt{s} \approx 1.19968$ is the unique solution of $\beta = \coth \beta$. 
It is known that the critical catenoid is a  free boundary minimal surface \cite{Fraser_2015}.

\subsubsection{Higher eigenvalues} \label{s:HigherEigs} For larger values of $j$, we numerically solve \eqref{e:ExtStekCenAnn}. In Figure \ref{f:StekEigCenAnn2}, we plot the value of $\tilde \sigma_j$ as a function of $s$ and $\rho_s/ \rho_1$ for $j= 1, \ldots, 6$. The maximum value of $\sigma_j L$ is indicated and data for the maximum values is also tabulated.  Observe that for $j=1, 3, \ldots, 6$, we have that 
$ s\rho_s = \rho_1$.

\begin{figure}[t]
\begin{center}
\includegraphics[height=0.24\textwidth]{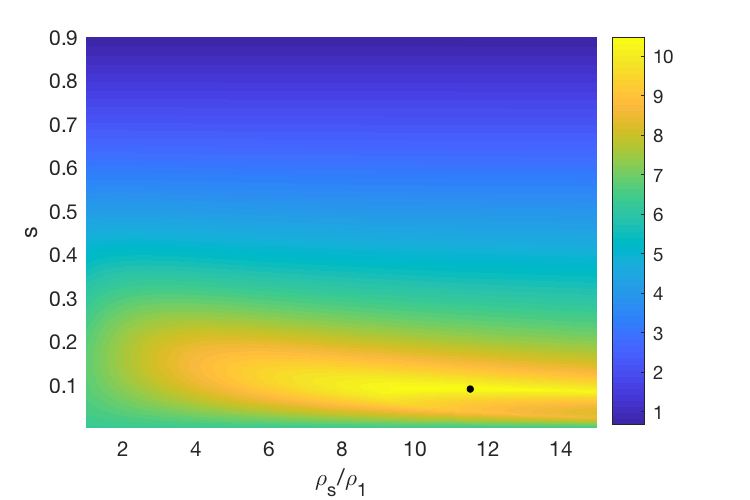}\includegraphics[height=0.24\textwidth]{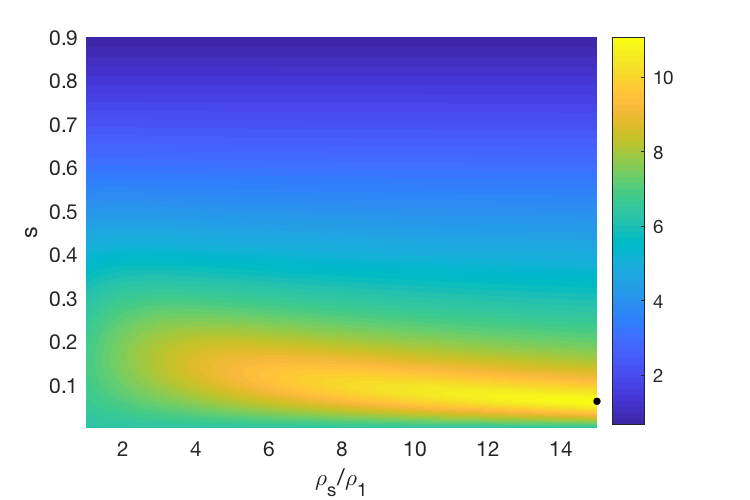}\includegraphics[height=0.24\textwidth]{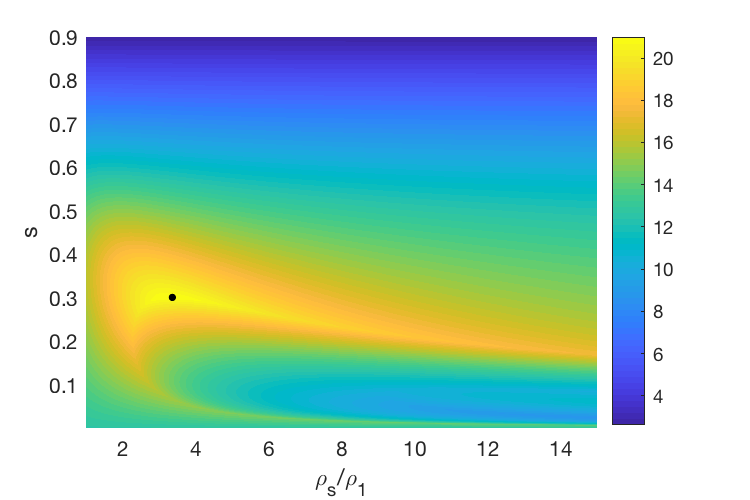}\\
\includegraphics[height=0.24\textwidth]{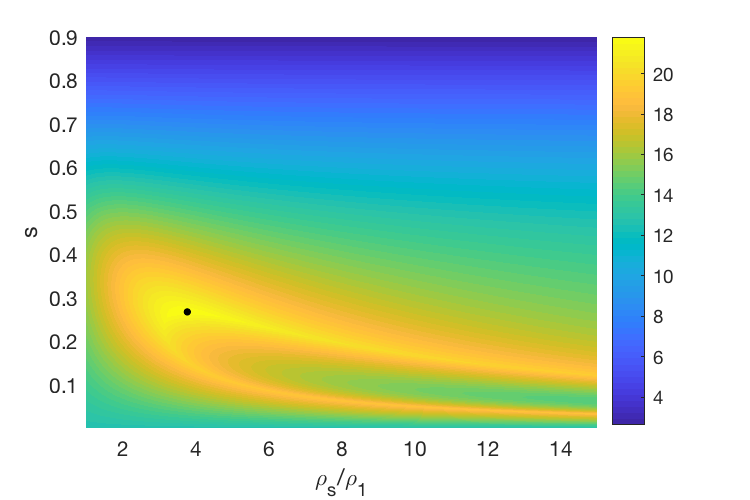}\includegraphics[height=0.24\textwidth]{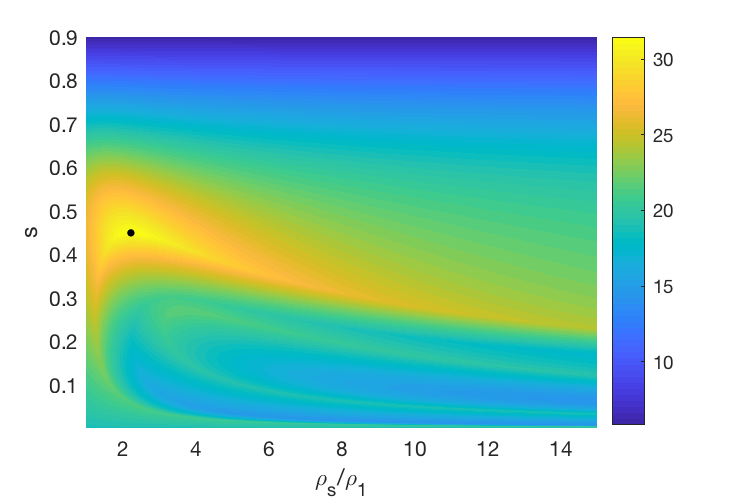}\includegraphics[height=0.24\textwidth]{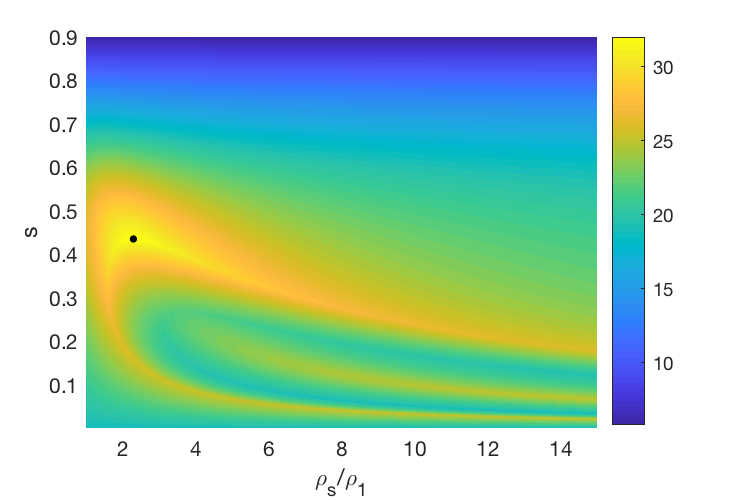} \\ \ \\
\begin{tabular}{|c|c|c|c|c|}
\hline 
$j$ & $\sigma_{j}L$ & $s$ & $\rho_{s}/\rho_{1}$ & multiplicity\tabularnewline
\hline 
\hline 
$1$ & $10.4748$ & $0.0908$ & $11.0161$ & 3\tabularnewline
\hline 
$2$ & $4 \pi$ & $0$ & $\infty$ & 3\tabularnewline
\hline 
$3$ & $20.9496$ & $0.3013$ & $3.3180$ & $3$\tabularnewline
\hline 
$4$ & $21.7656$ & $0.2679$ & $3.7322$ & 4\tabularnewline
\hline 
$5$ & $31.4243$ & $0.4494$ & $2.2251$ & 3\tabularnewline
\hline 
$6$ & $31.9495$ & $0.4354$ & $2.2988$ & 4\tabularnewline
\hline 
\end{tabular}
\end{center}
\caption{{\bf (top)} The value of $\tilde \sigma_j = \sigma_j L $ for $s\in [0.001,0.9]$ and $\frac{\rho_s}{\rho_1} \in [1, 15]$ for  $j=1,\ldots 6$. The black dots indicate the maximum values in the domain that is shown. {\bf (bottom)} A table with the maximum values of $\tilde \sigma_j$, the values of $s$ and  $\frac{\rho_s}{\rho_1}$ attaining the maximum, and the multiplicity of the eigenvalue at the maximum. See Section~\ref{s:HigherEigs}.}
\label{f:StekEigCenAnn2}
\end{figure}

For odd $j = 2m-1$, $m \in \mathbb N$, from the results of Fan, Tam, and Yu \cite{Fan_2014}, we have that the extremum is attained at the crossings of the  two length-normalized eigenvalues with associated $L^2(\partial \Omega, \rho)$-normalized eigenfunctions
\begin{align*}
\sigma_{j,-} L &= 4 \pi  j \frac{1-s^j }{1+s^j}, && 
u(r,\theta) =\frac{1}{\sqrt{2 \pi \rho} }  \frac{ \cosh\left( j \log \frac{r}{\sqrt s} \right) }{ \cosh\left( j \log \sqrt s \right) } \{ \cos j  \theta,  \ \sin  j \theta \} \\
\sigma L &= \frac{8 \pi}{ \log s^{-1}},  && 
u(r,\theta) =  \frac{1}{\sqrt{4 \pi \rho } } \frac{\log \frac{r}{\sqrt{s}}}{ \log \sqrt{s}} . 
\end{align*}
The two values of $\sigma L$ are equal when $s$ is the unique solution of the transcendental equation 
$$
\frac{1+s^j}{ 1-s^j } =-  \log s^{\frac j 2}, 
\qquad \qquad s > 0. 
$$
We obtain $\tilde \sigma_{2m - 1} = m \tilde \sigma_1^\star $, for $m \geq 1$. 
The extremal metric is achieved by the $m$-fold cover of the critical catenoid.

For even $j$, Fan, Tam, and Yu \cite{Fan_2014} show the following. 
For $j=2$, the extremal value is not attained among rotationally symmetric annuli 
and for even $j \geq 4$, the extremal value is attained. 
For $m \geq 2$, we have $\sigma_{2m} L = 4 m \pi \tanh(\frac{m T_{m,1}(1)}{2})$,  where $T_{m,1}(1)$ is the unique positive root of $m \tanh \frac{m s}{2} \tanh \frac{s}{2}= 1$
The extremal metric is achieved by the critical $m$-M\"{o}bius band, which have  genus $\gamma = 1$. 
These are not in the class of surfaces relevant to our later computational examples.

\section{Computational Methods} \label{s:CompMeth}
In Section~\ref{s:Red}, we described how conformal maps could be used to reduce  the general Steklov eigenproblem \eqref{e:Steklov} to the Euclidean Steklov eigenproblem \eqref{e:Steklov2}. 
In this section, we describe the computational methods used to solve the Euclidean Steklov eigenproblem \eqref{e:Steklov2}, 
optimization methods used to solve the extremal eigenvalue problem \eqref{e:MaxSteklov}, and
methods for computing the minimal surface from the Steklov eigenfunctions.

\subsection{Solving the Euclidean  Steklov eigenproblem \eqref{e:Steklov2}}
We use the method of particular solutions to solve the Steklov eigenproblem \eqref{e:Steklov2}. This method for multiply-connected Laplace problems was recently discussed in  \cite{Trefethen_2018}. The methods rely on the following Theorem.  

\begin{thm}[Logarithmic Conjugation Theorem \cite{Trefethen_2018}] 
\label{t:LogConjThm}
Suppose $\Omega$ is a finitely connected region, with $K_1,\ldots, K_N$ denoting the bounded components of the complement of $\Omega$. For each $j$, let $a_j$ be a point in $K_j$. 
If $u$ is a real valued harmonic function on $\Omega$, then there exist an analytic function $f$ on $\Omega$ and real numbers $c_1, \ldots , c_N$ such that
$$
u(z) = \textrm{Re} f(z) + c_1 \log |z-a_1| + \cdots + c_N \log |z-a_N|, 
\qquad \qquad   
\forall z \in \Omega.
$$
\end{thm}

Let $M \in \mathbb{N}^*$ and consider some fixed punctured disk $\Omega_{c,r}$. 
Based on Theorem~\ref{t:LogConjThm}, we define the finite basis 
$\mathcal{B}$ to approximate solutions of eigenvalue problem \eqref{e:Steklov2} 
as the  union of the harmonic rescaled real and imaginary parts of the functions
\begin{equation} \label{e:mpsbasis}
\mathcal{B} =  \bigcup_{j = 0}^M \left\{ z \mapsto z ^ j \right\}
                \bigcup_{i = 1}^{k-1} \bigcup_{j = 1}^M \left\{  z \mapsto  \frac{1}{(z - c_i) ^ j}   \right\}
                \bigcup_{i = 1}^{k-1} \left\{ z \mapsto \log|z- c_i| \right\}. 
\end{equation}
For instance, we rescaled the basis polynomial $Re \left( \frac{1}{(z - c_2)^3} \right)$ 
 by a factor $r_2^3$ so that this basis function takes values of order $1$  
  on the second circle.
Consider now $(p_l)_{1 \leq l \leq L}$ a uniform sampling with respect to arc 
length of $\partial \Omega_{c,r}$. Using $\mathcal{B}$, we approximate solutions 
of eigenvalue problem \eqref{e:Steklov2b}  by the solution of the non symmetric
 square generalized eigenvalue problem
\begin{equation} \label{e:mpseig}
B^T A  \ u_d =  \sigma_d  \ B^T B \  u_d, 
\end{equation}
where
$ A = \left(\frac{\partial \phi}{\partial n}(p_l) \right)_{1 \leq l \leq L, \  \phi \in \mathcal{B}}$  and 
$ B = \left( \phi(p_l) \right)_{1 \leq l \leq L, \ \phi \in \mathcal{B}}$.

\begin{ex} \label{e:ex4}
To illustrate the complexity of the approach to obtain a fine approximation of eigenvalues, we considered a circular domain  with four holes and $L = 5000$ points; see Figure~\ref{f:diskconv}(left).
We evaluated the first six nontrivial  eigenvalues with a high number of $\mathcal{B}$ elements for $M = 50$. In Figure~\ref{f:diskconv}(right), you can observe the evolution of the error with respect to $M$ for $M$  taking values from $2$ to $10$.
 Taking the converged values as an approximation of the exact ones, in this specific example, it can be observed that with $M=10$ the error is already smaller than $10^{-8}$. 
 Here, the first nontrivial eigenvalue has multiplicity two, so the curves are almost indistinguishable. 
 \end{ex}
 
\begin{figure}
    \centering
    \includegraphics[height=0.35\textwidth]{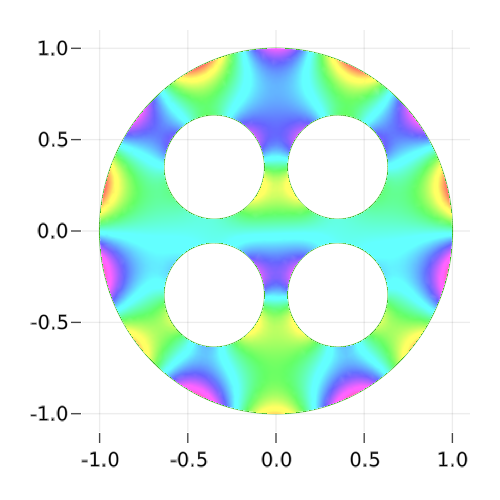} 
    \hspace{1cm}
    \includegraphics[height=0.35\textwidth]{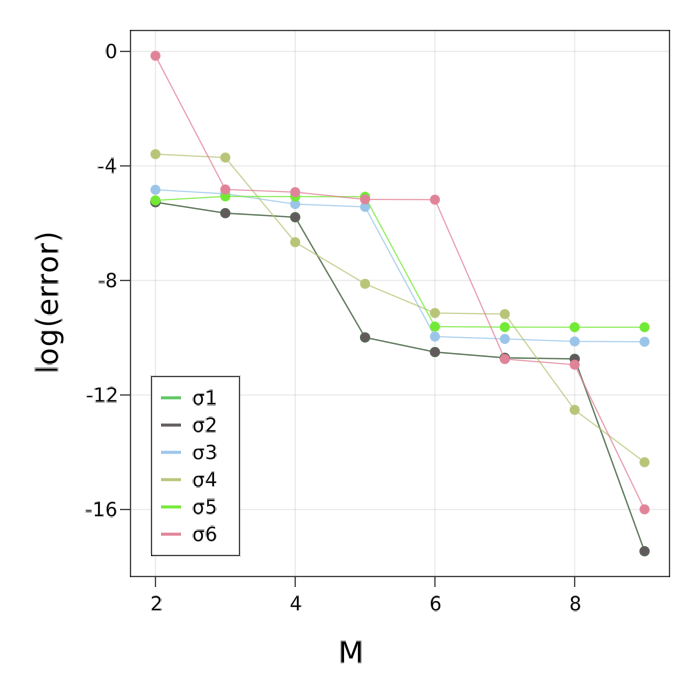} 
    \caption{An illustration of the convergence of the eigenvalues with respect to 
     the number of basis functions for a non-simply connected domain. See Example~\ref{e:ex4}. }
    \label{f:diskconv}
\end{figure}

\begin{ex} \label{e:ex5}
We now consider a geometric convergence study related to Example~\ref{e:Hippopede}; see also Figure~\ref{f:Hippopede}. 
Using the mapping from the unit disk to the
Hippopede domain, $\Omega_\alpha$, we study the limit as $\alpha \to 0$. 
Our computations are performed on the unit disk with non-constant density, $\rho$, as given  in Example~\ref{e:Hippopede}. 
In the limit, the density becomes singular, and the purpose of this example is to illustrate that a weakness of our numerical method is that we cannot accurately compute  eigenvalues of pinched domains ($\alpha \to 0$) or, equivalently, if the density is singular. 
The results are displayed in Table~\ref{t:AlphaConv}. 
The values for the disjoint union of two radius 0.5 disks, obtained in the limit $\alpha \to 0$, are given in the rightmost column of  Table~\ref{t:AlphaConv}. 
We note a very slow convergence of the eigenvalues as $\alpha \to 0$. 
\end{ex}

\begin{table}[t!]
\begin{center}
\begin{tabular}{l | l  l l | l }
$j$ & $\alpha = 0.1$ & $\alpha = 0.06$ & $\alpha = 0.04$ & $\alpha = 0$ \\
\hline
1 & 0.37968380 & 0.32288183 & 0.28797139 & 0  \\
2 & 1.99258587 & 1.99688224 & 1.99338590 & 2 \\
3 & 2.02351398 & 2.00917719 & 1.99906424 & 2 \\
4 & 2.20444005 & 2.66795651 & 2.09627138 & 2 \\
5 & 2.78126086 & 2.66795651 & 2.60980134 & 2 \\
6 & 3.99885096 & 3.99479457 & 3.98132439 & 4\\
7 & 4.09199872 & 4.03602674 & 4.00214005 & 4\\
8 & 4.36831843 & 4.24271684 & 4.18039135 & 4\\
9 & 4.95936215 & 4.80367369 & 4.69676874 & 4\\
10& 6.02510373 & 6.00554908 & 6.01439273 & 6
\end{tabular}
\end{center}
\label{t:AlphaConv}
\caption{The first ten nontrivial Steklov Eigenvalues, $\sigma_j$, of the Hippopede domain, $\Omega_\alpha$, for $\alpha = 0.1$, $0.06$, $0.04$. The last column are the values, known analytically, that appear in the limit as $\alpha \to 0$. }
\end{table}%

\subsection{Optimization methods for extremal Steklov eigenvalues \eqref{e:MaxSteklov2}}
We used gradient-based optimization methods to solve the extremal Steklov eigenvalue problem \eqref{e:MaxSteklov2}. We first describe our parameterization of the boundary 

\subsubsection{Parameterizing the geometry}
Let $\rho \in L^\infty( \partial \Omega_{c,r})$ be the boundary density and 
denote the restriction of $\rho$ to the $i$-th disk boundary by  
$$\rho_i = \rho|_{ \partial D(c_i,r_i)}, 
\qquad \qquad i=1,\ldots,k-1. 
$$
Finally, denote $D_k := D$ and $\rho_k$ the restriction of $\rho$ to $\partial D_k$. 
Thus, if $\Omega_{c,r}$ has $b$ boundary components, the geometry is described by the parameters 
$$
\{c_i\}_{i=1}^{b-1}, 
\qquad 
\{r_i\}_{i=1}^{b-1}, 
\quad \textrm{and} \quad 
\{ \rho_i (x) \}_{i = 1}^b. 
$$
Since $\partial D(c_i,r_i) \cong \mathbb S^1$, we expand each $\rho_i$ in the truncated Fourier series
$$
\rho_i(\theta) = A_{i,0}+ \sum_{\ell=0}^N A_{i,\ell} \cos( \ell \theta) + B_{i,\ell} \sin(\ell \theta), \qquad \theta \in [0,2\pi].
$$

From Remark~\ref{r:dim}, it would be possible to  center one of the holes at the origin and another on the positive $x$-axis. However,  we found that the representation of the boundary density $\rho$ for finite basis size (finite $N$) was better without fixing these centers. 

\subsubsection{Gradient based optimization methods}
As in \cite{Akhmetgaliyev_2017}, to handle multiple eigenvalues, we trivially transform \eqref{e:MaxSteklov2} into the following problem 
\begin{subequations}
\label{e:MaxSteklov3} 
\begin{align}
\max \ &  t \\
\textrm{s.t.} \ & t \leq  \sigma_i L && i = j, j+1, \ldots,   
\end{align}
We approximated  the positivity constraint $\rho \geq 0$ by imposing the positivity 
 on all $L$ sample points, 
 \begin{equation} 
 \rho(p_\ell) \qquad \qquad \ell =1,\ldots, L. 
 \end{equation}
 This approximation leads to linear inequalities 
 with respect to the coefficients $(A_{i,l}, B_{i,l})$ only.
We also augment the previous optimization problem with the geometrical constraints 
in  \eqref{e:MaxSteklov2} by imposing the  (few) quadratic constraints on 
the variables $(c_i, r_i)_{1 \leq i \leq k - 1}$: 
    \begin{align}
    \label{e:geomconstra}
    &    |c_i|^2   < (1- r_i)^2 && i  = 1,\ldots, k-1, \\
    \label{e:geomconstrb}
   &     |c_i - c_j| ^ 2 > (r_i + r_j ) ^2 && i, \ j = 1,\ldots, k-1, \, j \neq i. 
    \end{align}
\end{subequations}
Using the derivatives computed in \eqref{e:Drho}, \eqref{e:Dr}, and \eqref{e:Dc}, 
together with the interior point method implemented in \cite{KNITRO_2006}, we solved
 \eqref{e:MaxSteklov3}.
 All results of section  \ref{s:NumResults}, have been obtained with the following 
 parameters: $M= 30$ (maximal order of basis elements), $L=10^4$ (number of sampling points)
  and at most $5,000$  iterations to reach a first order optimality condition 
  criteria  to a relative precision of $10^{-6}$. Observe that in all cases, we were 
   able to recover the multiplicity three of the optimal eigenvalue up to $6$ digits.

 In our implementation, the computational cost is proportional to the number of connected components of the boundary. 
 For instance, one hour of computation on a standard laptop was required to obtain  the desired precision for three boundary components.

\subsection{Computing the free boundary minimal surface from the Steklov eigenfunctions}
At this point we assume that we have successfully solved the extremal  Steklov problem \eqref{e:MaxSteklov2} and want to use Theorem~\ref{t:Fraser} to compute the associated free boundary minimal surface using the Steklov eigenfunctions. 

Let $\sigma$  denote the optimal eigenvalue and assume that it has multiplicity $n$. 
Define the mapping 
$v = [v_1, \ldots, v_n] \colon \Omega \to \mathbb R^n$, 
where $ \{ v_i \}_{i=1}^n$ is some choice of basis for the  $n$-dimensional eigenspace. 
 For $A \in \mathbb R^n$, we consider the map 
 $ u_A \colon \Omega \to R^n$, 
 defined by 
 $$
 u_A(x) = \left[v_1(x), \ldots, v_n(x) \right] A,
 \qquad \qquad x \in \Omega. 
 $$ 
We want to identify the matrix $A$ so that the map $u_A = u  =  [u_1, \ldots, u_n]$ satisfies the spherical and the isothermal coordinate conditions, 
\begin{subequations}
\label{e:IsoCoord}
 \begin{align}
& |\partial_r u(r,\theta) |^2 =  r^{-2} | \partial_\theta u(r,\theta) |^2, 
&& \forall (r,\theta) \in \Omega_{r,c} \\
& \partial_r u(r,\theta) \cdot r^{-1} \partial_\theta u(r,\theta)  = 0, 
&& \forall (r,\theta) \in \Omega_{r,c}.
\end{align}
\end{subequations}
To identify the matrix $A$, so that $u_A \colon \Omega \to \mathbb R^n$ satisfies \eqref{e:IsoCoord},  we construct the objective function 
\begin{equation} 
\label{e:EvolEq}
J(A) = \int_{\partial \Omega} W( u_A(x) ) \ dx + 
\int_{\Omega}  \left( |\partial_r u_A(r,\theta) |^2 -   r^{-2} | \partial_\theta u_A(r,\theta) |^2 \right) ^2 + | \partial_r u_A(r,\theta) \cdot r^{-1} \partial_\theta u_A(r,\theta) |^2 \ dx,
\end{equation}
 where $W(u) = \frac{1}{4}(|u|^2-1)^2$. We then minimize $J(A)$ over $A \in \mathbb R^{n \times n}$. 
   In all experiments in section \ref{s:NumResults},  using this selection process, we were able 
   to obtain three eigenfunctions which take values in the sphere on $ \partial \Omega$
  to an absolute pointwise error bounded by $10^{-3}$. Moreover, since we have a parameterization of the surface, using the well-known analytic formula, we were  able to 
   compute the mean curvature of the  surfaces, which 
   in all cases was bounded by $10^{-2}$. The mean curvature and the Gaussian curvature are plotted on the free boundary minimal surface at \cite{EdouardWebpage}. 
Additionally, the angle that the boundary makes with the normal vector to the sphere is less than one degree.


\begin{figure}
    \centering
    \includegraphics[width=\textwidth]{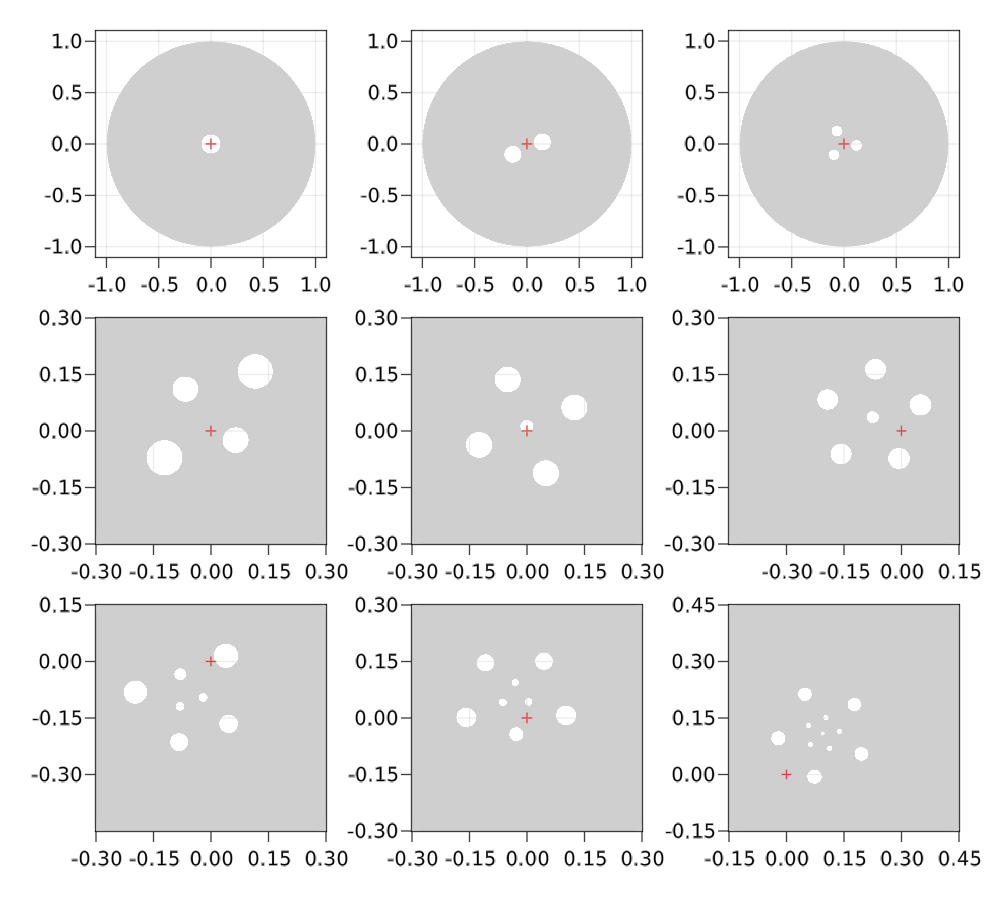} 
    \caption{Optimal disks configurations for $2$ to $9$ and $12$ (last bottom right picture) 
             connected components of the boundary. The red cross indicates the center 
             of the unit disk.}
    \label{f:alldisks}
\end{figure}

\begin{figure}
    \centering
    \includegraphics[width=\textwidth]{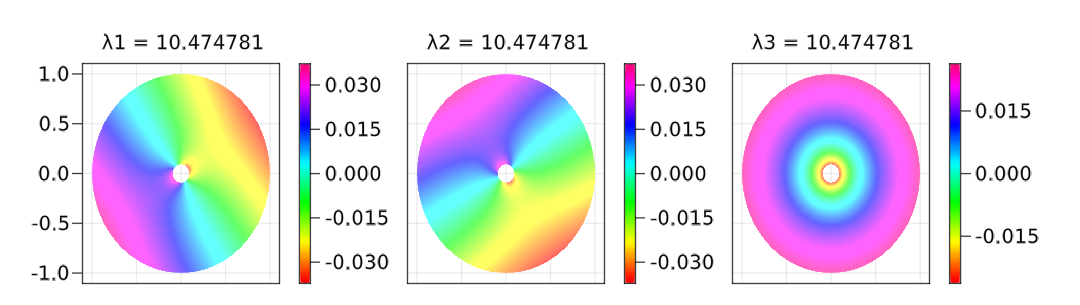} 
    \includegraphics[width=\textwidth]{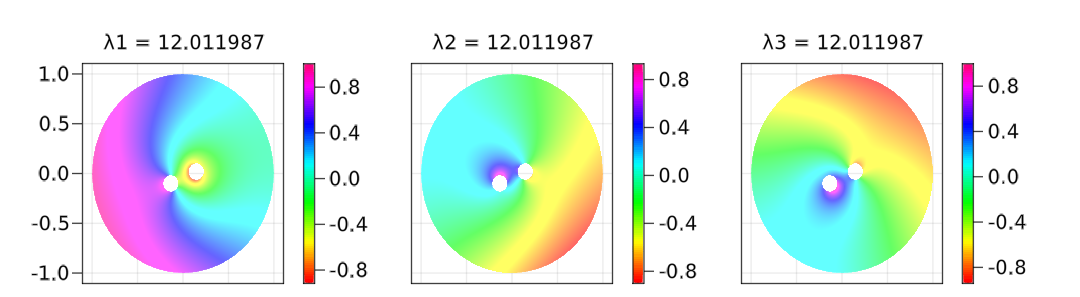} 
    \caption{Three linearly independent eigenfunctions associated to the first eigenvalue for 
             two and three boundary components.}
    \label{f:eigs23}
\end{figure}

\begin{figure}
    \centering
    \includegraphics[width=\textwidth]{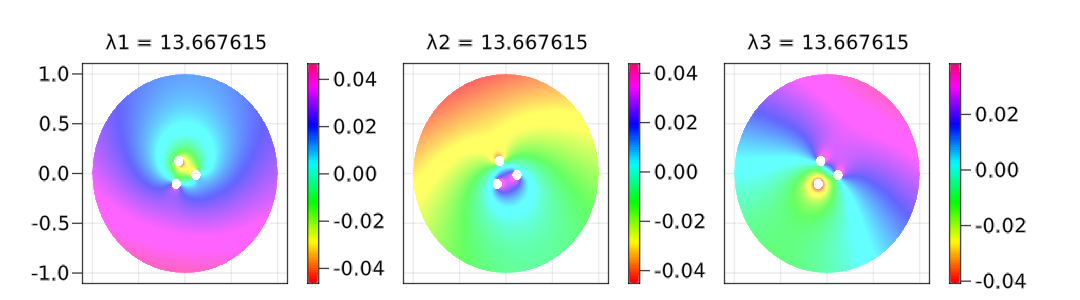} 
    \includegraphics[width=\textwidth]{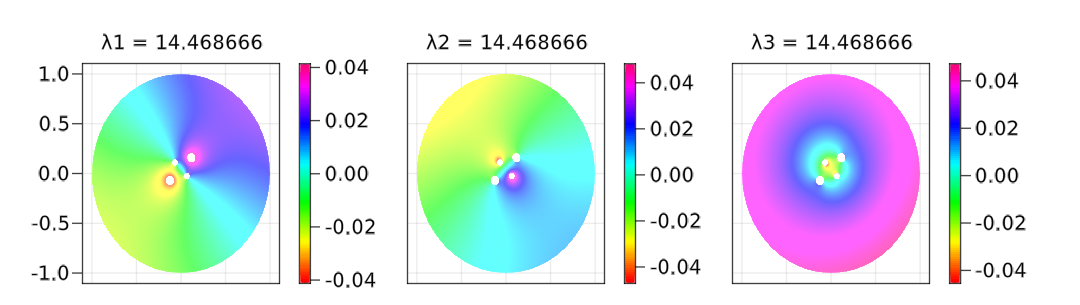} 
    \caption{Three linearly independent eigenfunctions associated to the first eigenvalue for 
             four and five boundary components.}
    \label{f:eigs45}
\end{figure}

\begin{figure}
    \centering
    \includegraphics[height=0.3\textwidth]{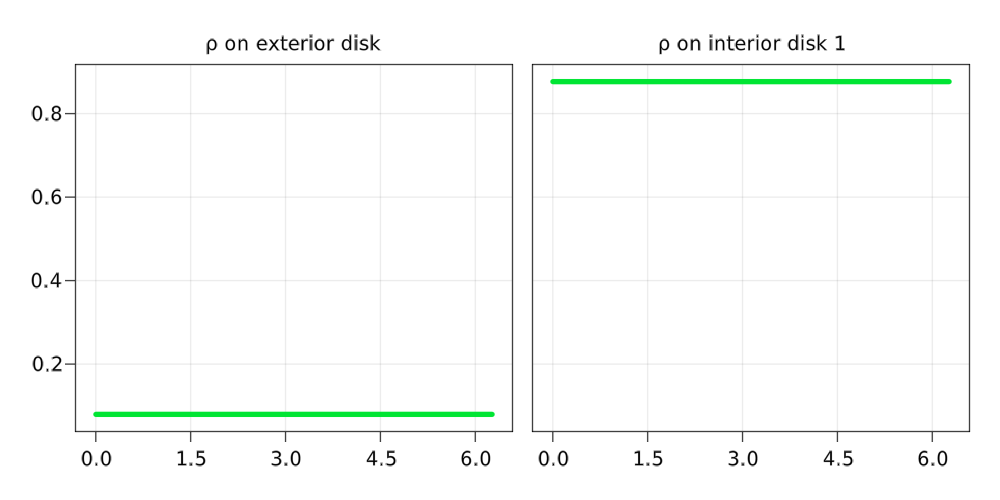} 
    \includegraphics[height=0.3\textwidth]{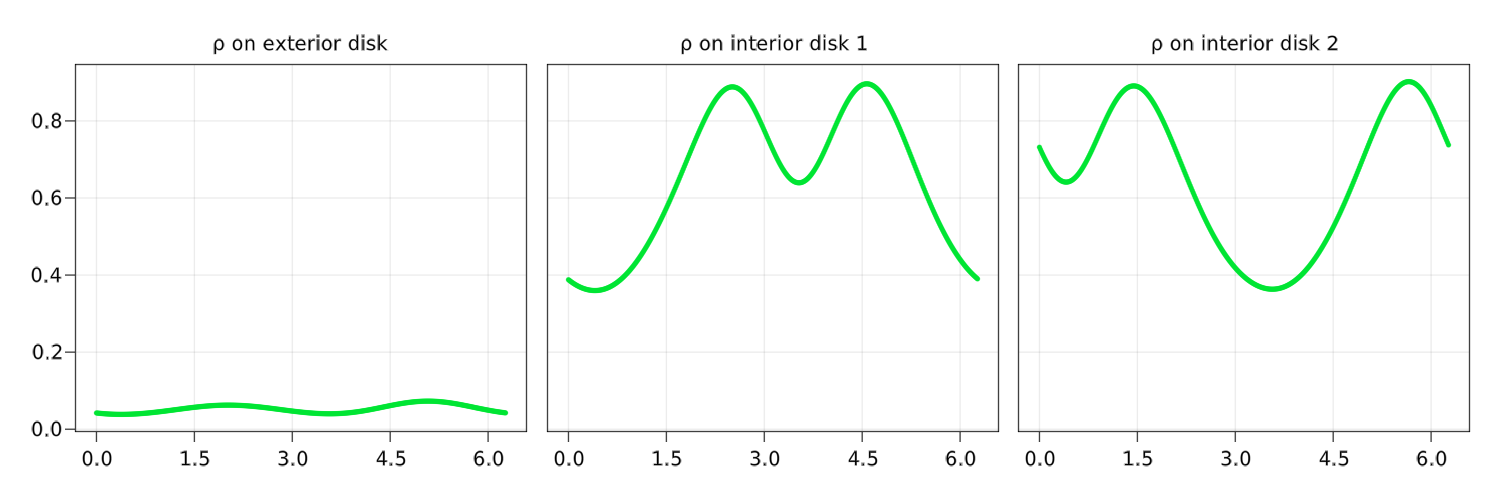} 
    \caption{Optimal densities for two and three boundary components.}
    \label{f:rho23}
\end{figure}

\begin{figure}
    \centering
    \includegraphics[height=0.6\textwidth]{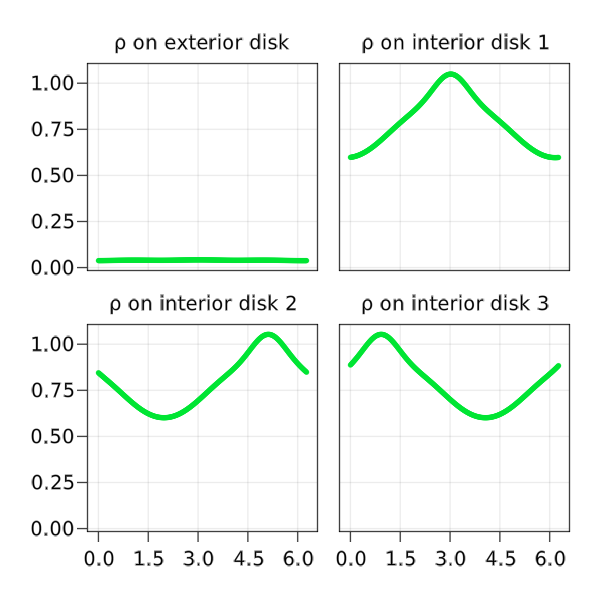} 
    \caption{Optimal densities for four boundary components.}
    \label{f:rho4}
\end{figure}

\begin{figure}
    \centering
    \includegraphics[height=0.6\textwidth]{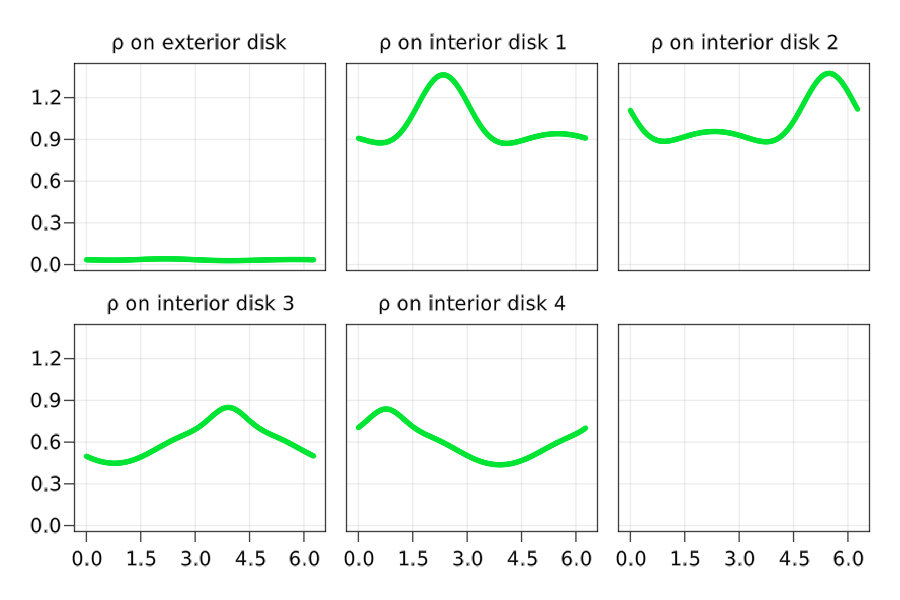} 
    \caption{Optimal densities for five boundary components. }
    \label{f:rho5}
\end{figure}

\begin{figure}
    \centering
    \includegraphics[width=0.3\textwidth]{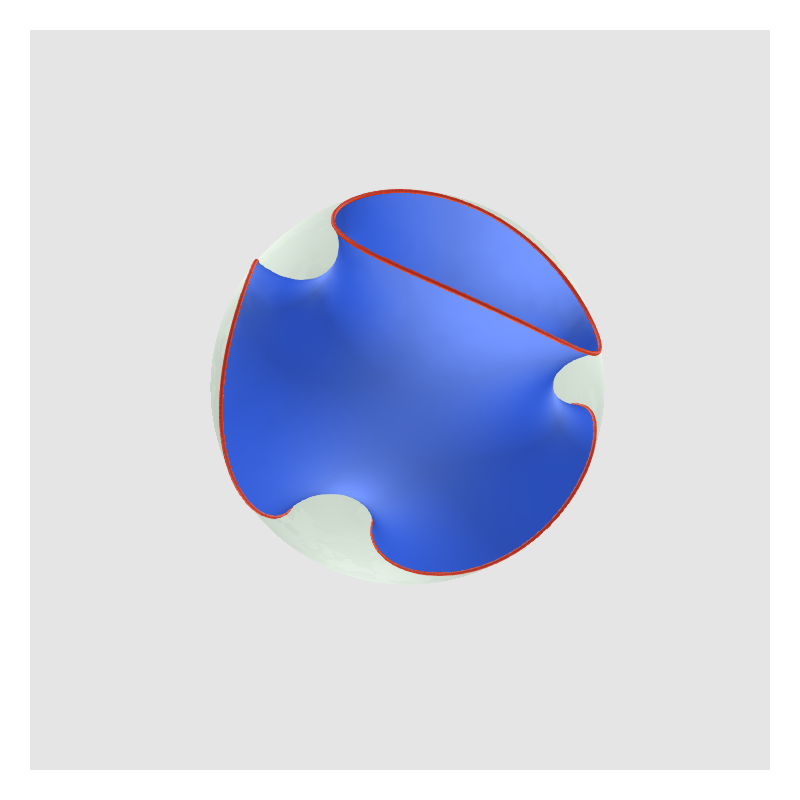} 
    \includegraphics[width=0.3\textwidth]{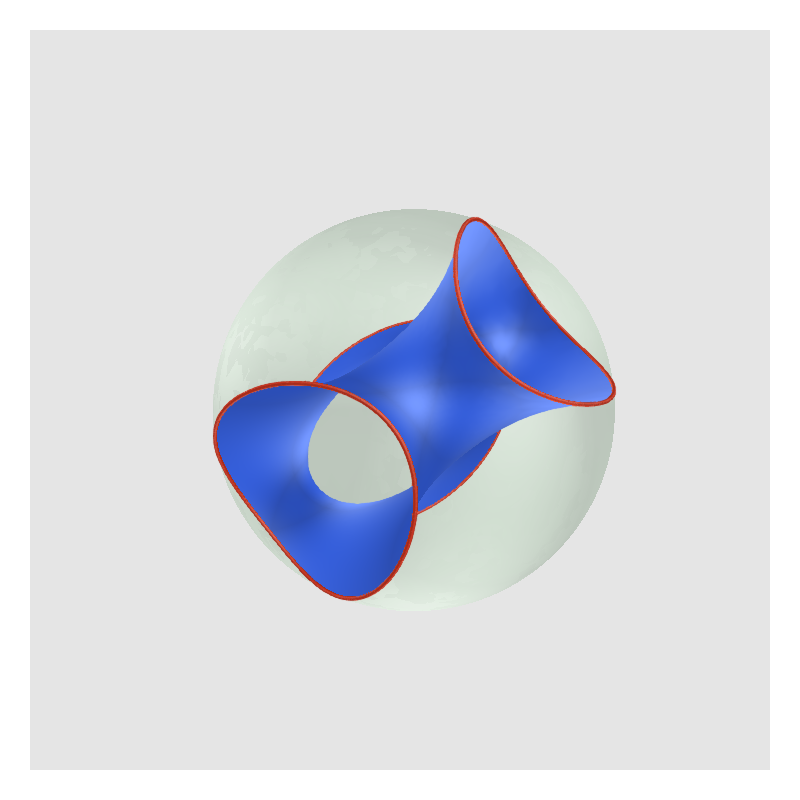} 
    \includegraphics[width=0.3\textwidth]{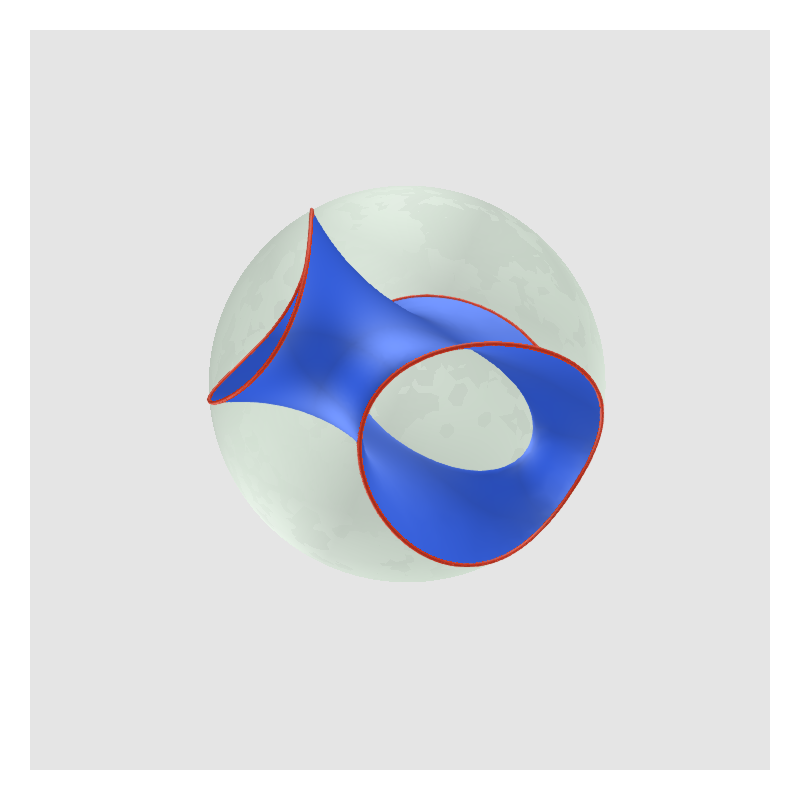}  \\
    \includegraphics[width=0.3\textwidth]{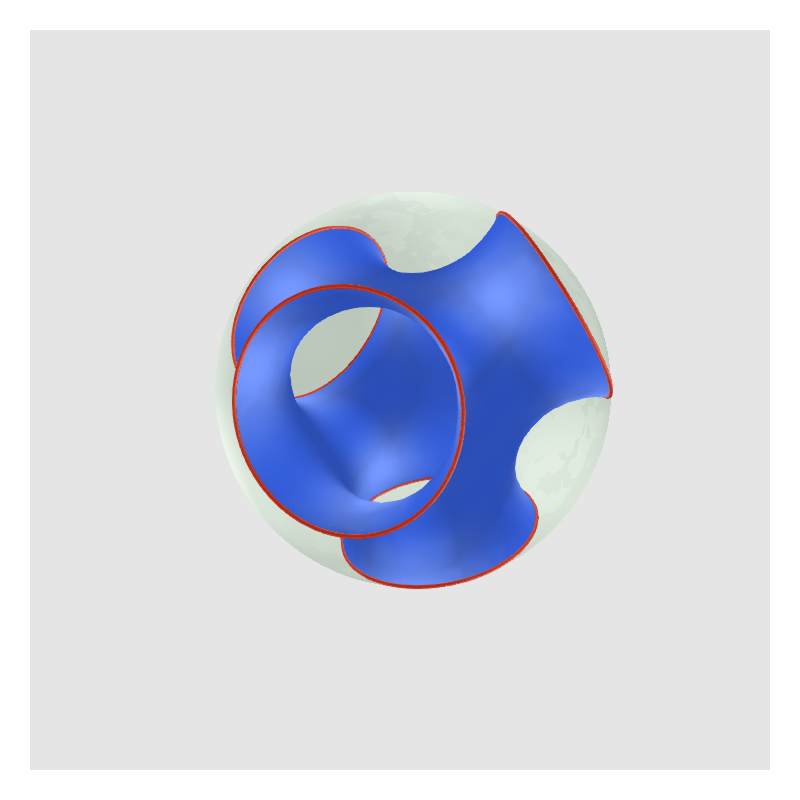} 
    \includegraphics[width=0.3\textwidth]{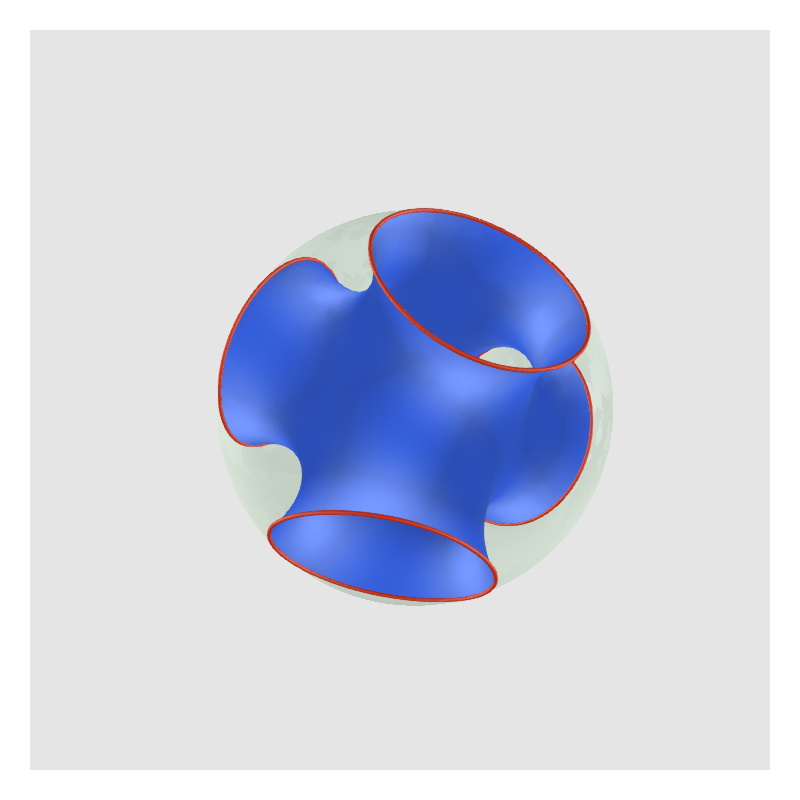} 
    \includegraphics[width=0.3\textwidth]{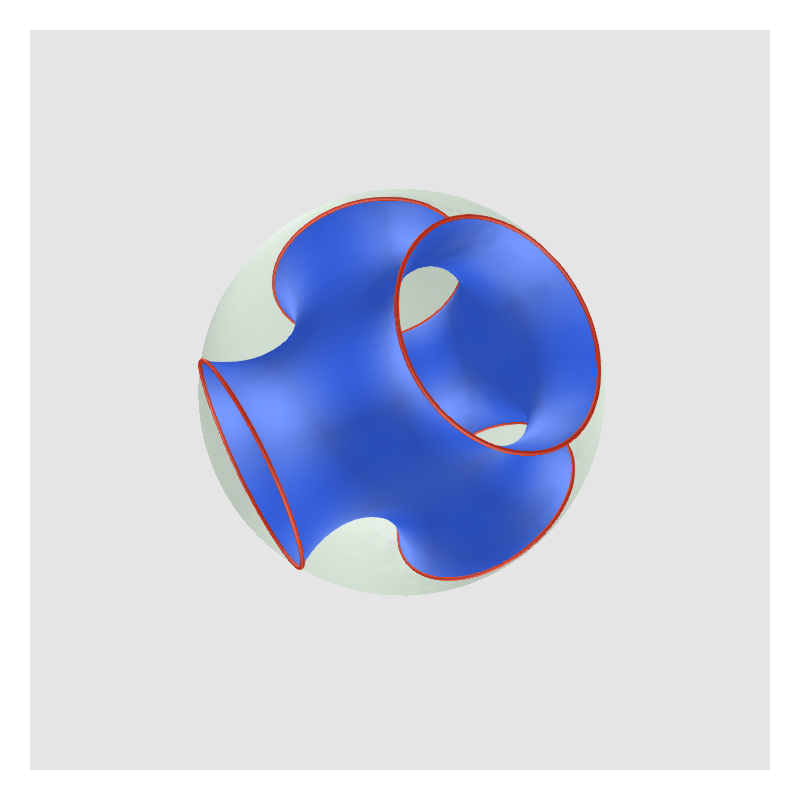}  
    \caption{Approximation of a minimal surface in the ball with three and four connected 
    components of the boundary.}
    \label{f:surf34}
\end{figure}

\begin{figure}
    \centering
    \includegraphics[width=0.3\textwidth]{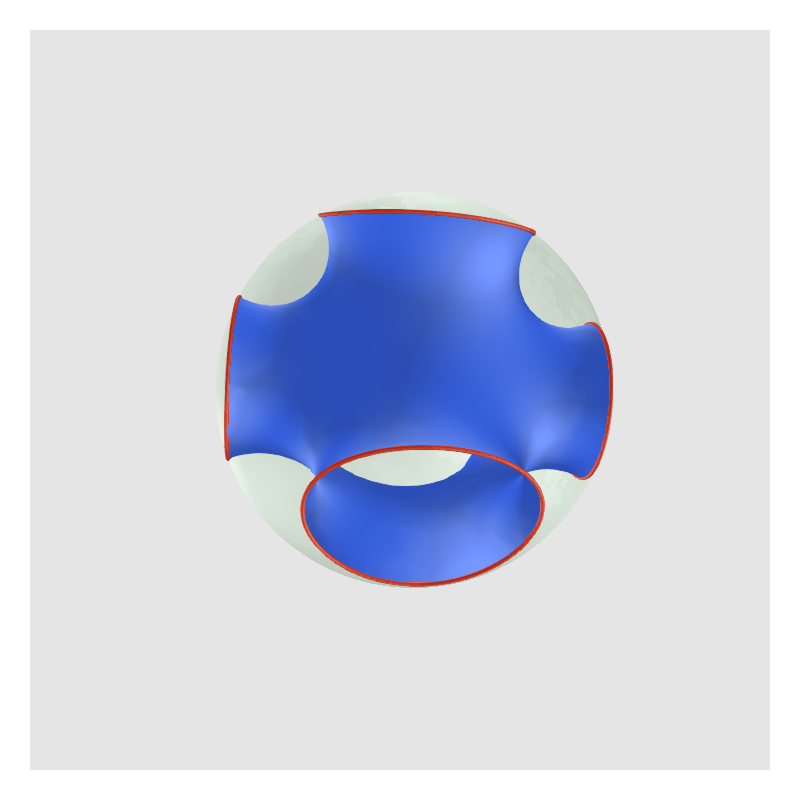} 
    \includegraphics[width=0.3\textwidth]{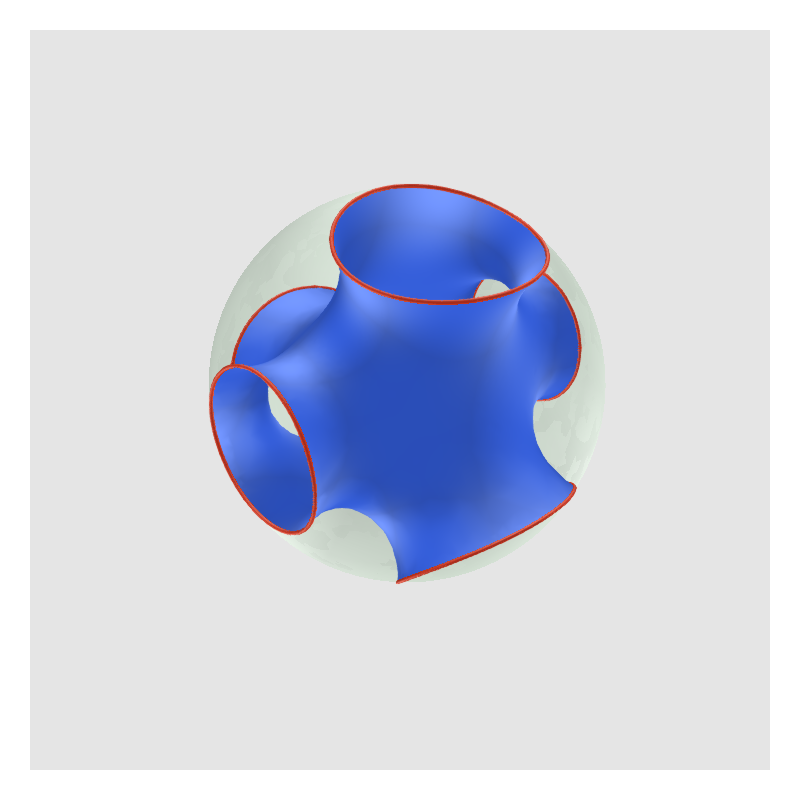} 
    \includegraphics[width=0.3\textwidth]{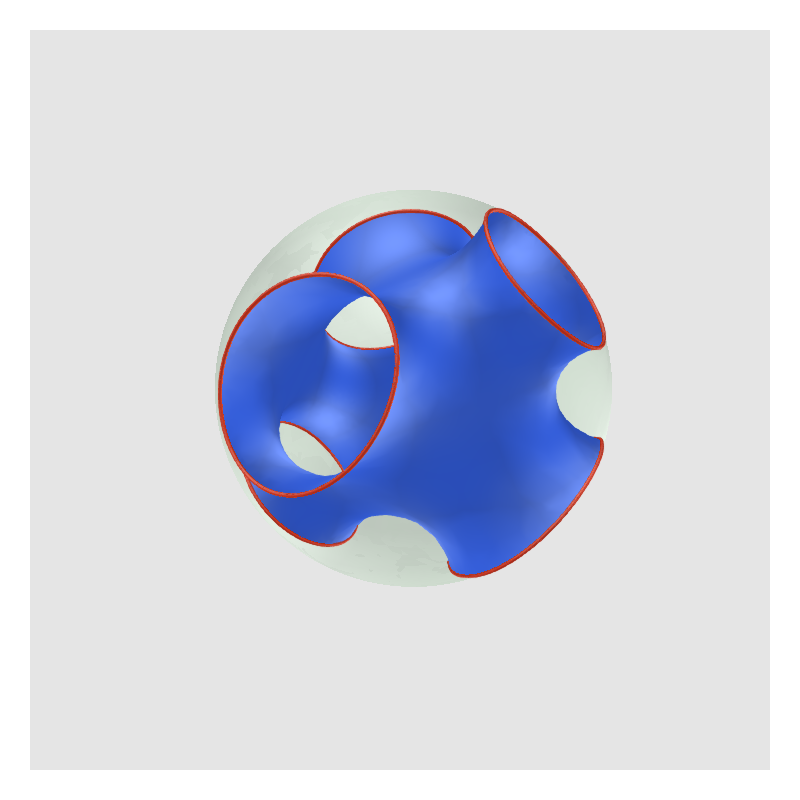}  \\
    \includegraphics[width=0.3\textwidth]{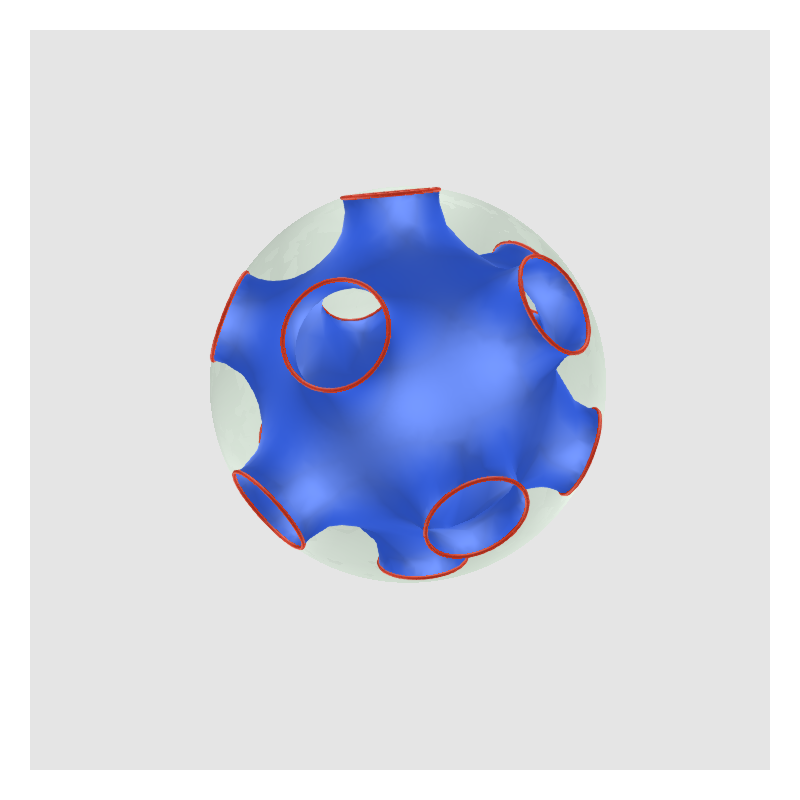} 
    \includegraphics[width=0.3\textwidth]{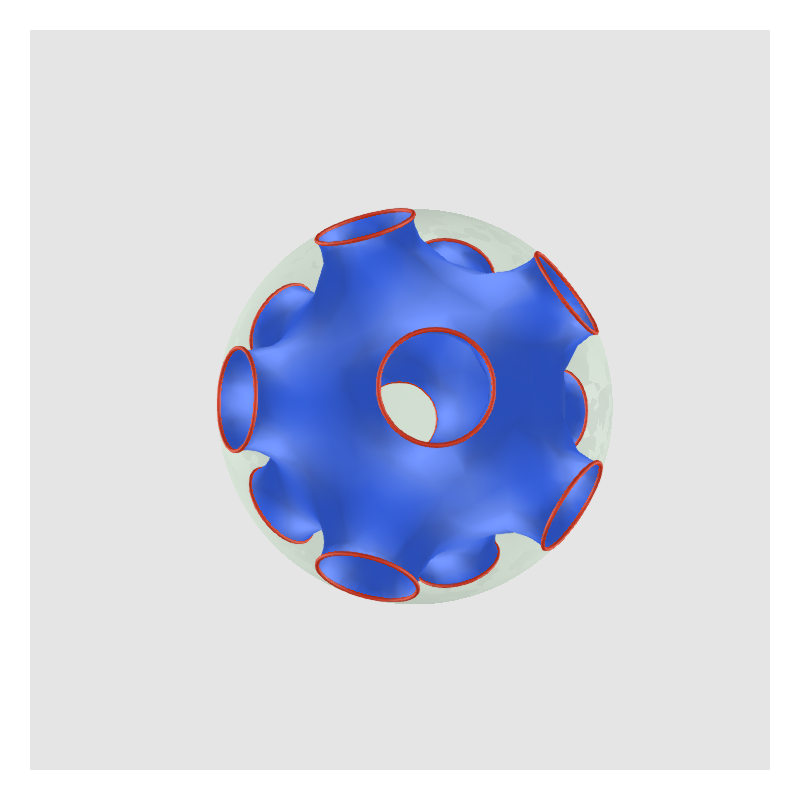} 
    \includegraphics[width=0.3\textwidth]{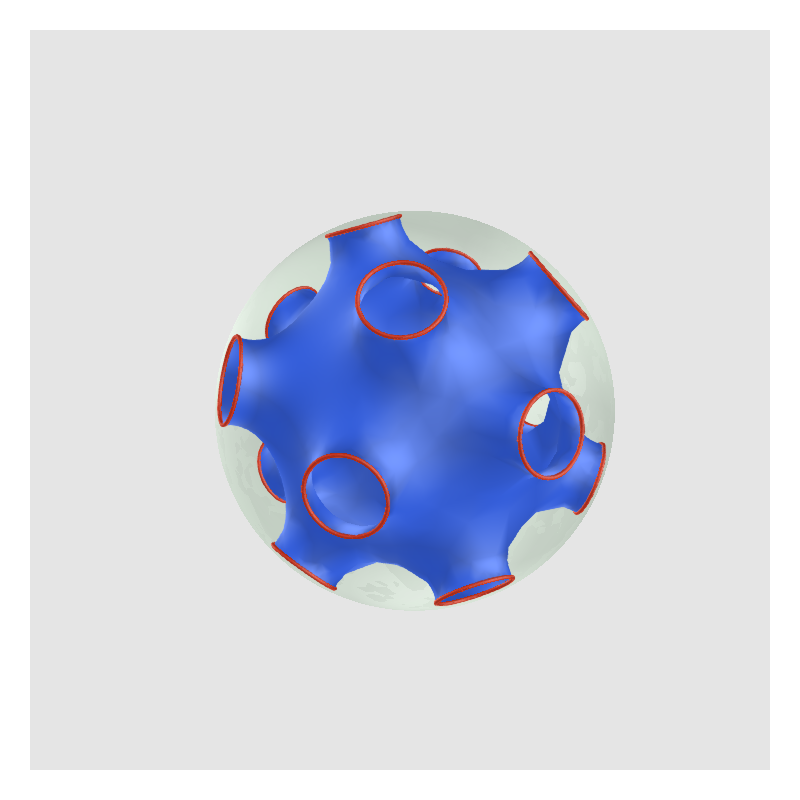}  
    \caption{Approximation of a minimal surface in the ball with five (first row),
     twelve (second row, two first views) and fifteen connected components of 
     the boundary.} 
    \label{f:surf51215}
\end{figure}

\begin{figure}
    \centering
    \includegraphics[width=0.3\textwidth]{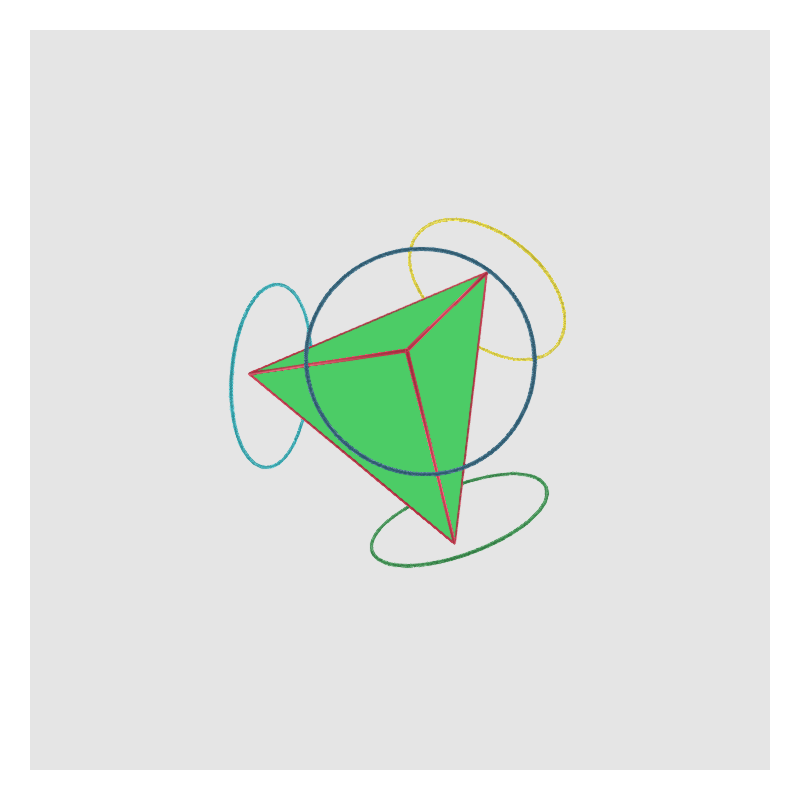} 
    \includegraphics[width=0.3\textwidth]{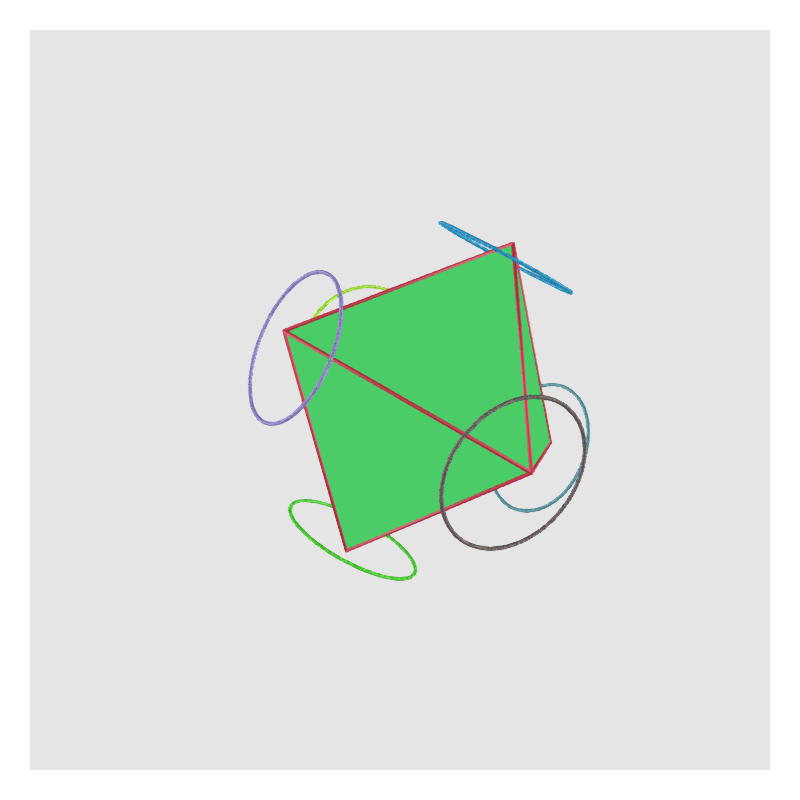} 
    \includegraphics[width=0.3\textwidth]{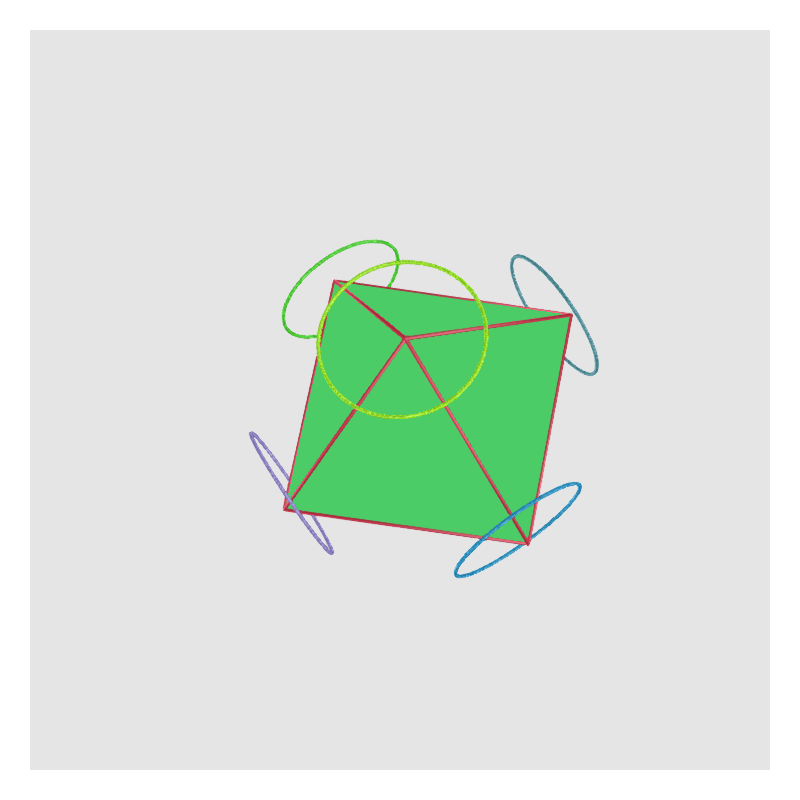}\\    
    \includegraphics[width=0.3\textwidth]{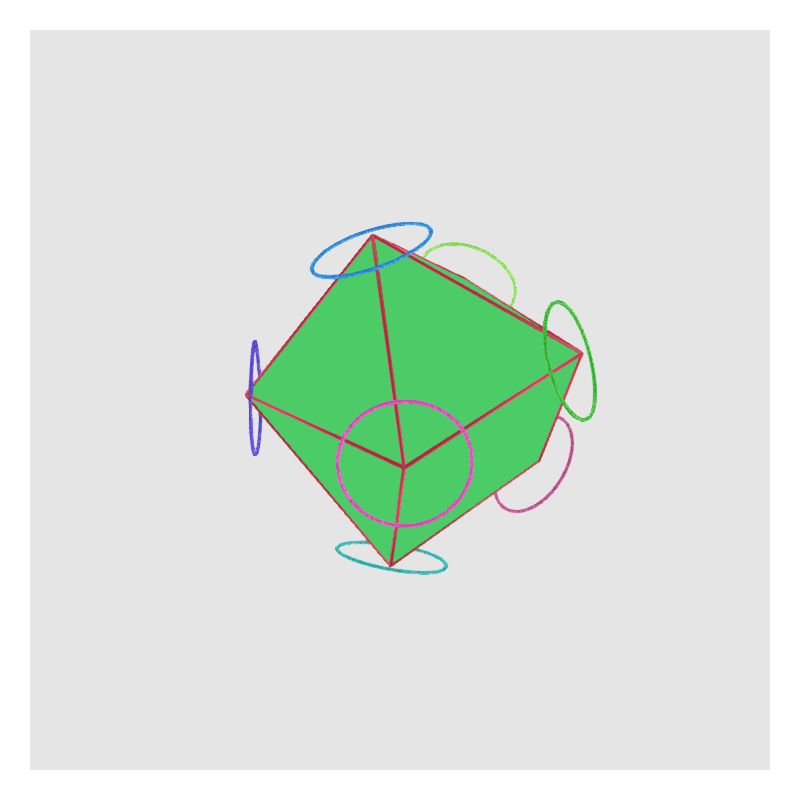} 
    \includegraphics[width=0.3\textwidth]{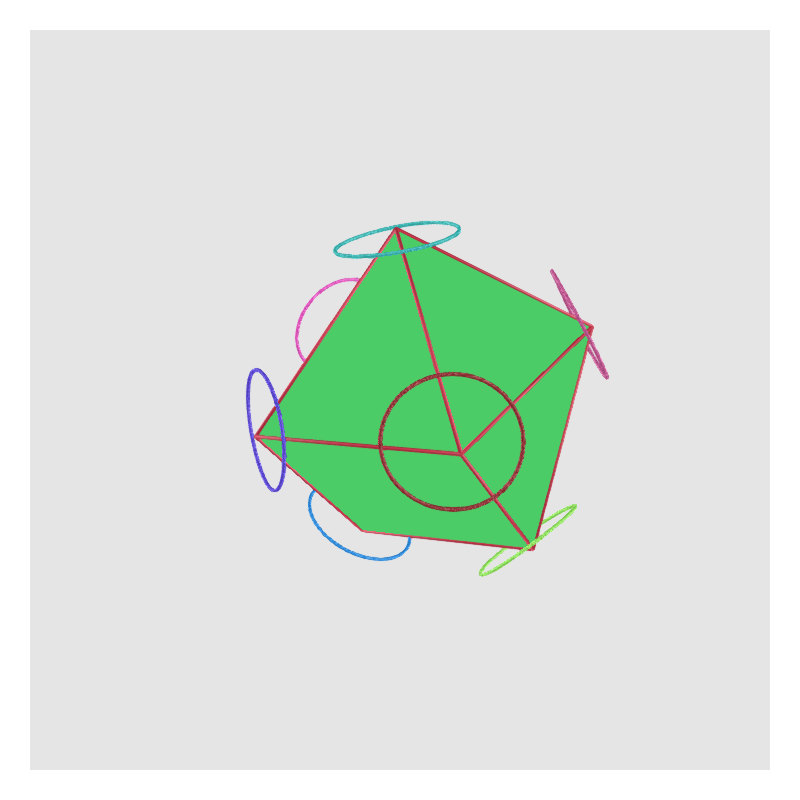} 
    \includegraphics[width=0.3\textwidth]{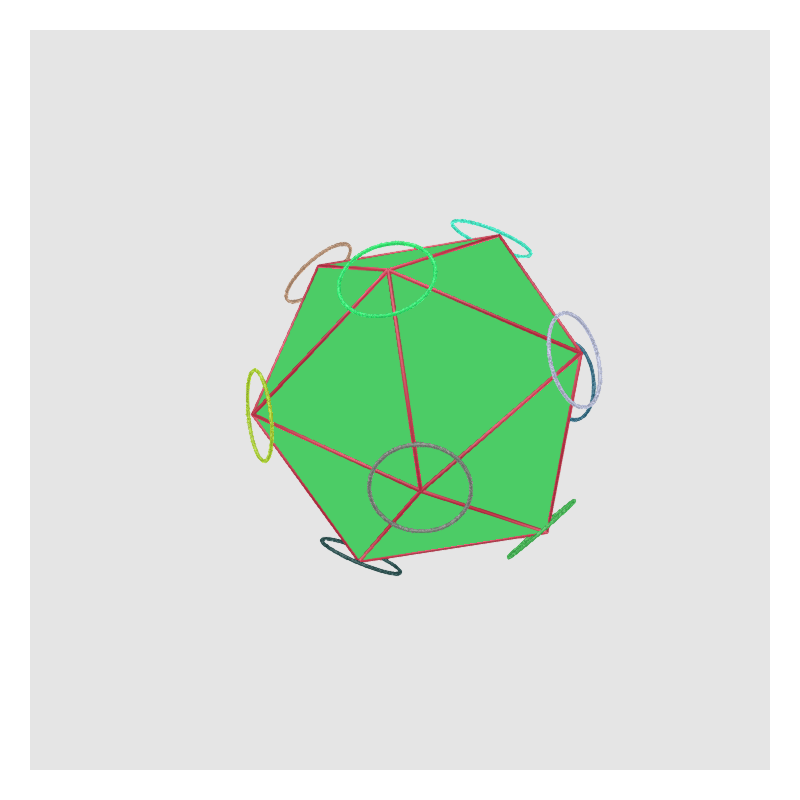}
    \caption{Convex polytopes associated to the center of mass of boundary connected 
    components of minimal surfaces in the ball. {\bf First row.} Four (first plot) 
    and six boundary connected components (the two remaining plots). 
    {\bf Second row.} Two views of a square antiprism associated to a
    minimal surface with a boundary made of height connected components
    and an icosahedron associated to a minimal surface with 
     twelve connected components in its boundary (last plot).}
    \label{f:platonic}

\end{figure}

\begin{figure}
    \centering
    \includegraphics[width=\textwidth]{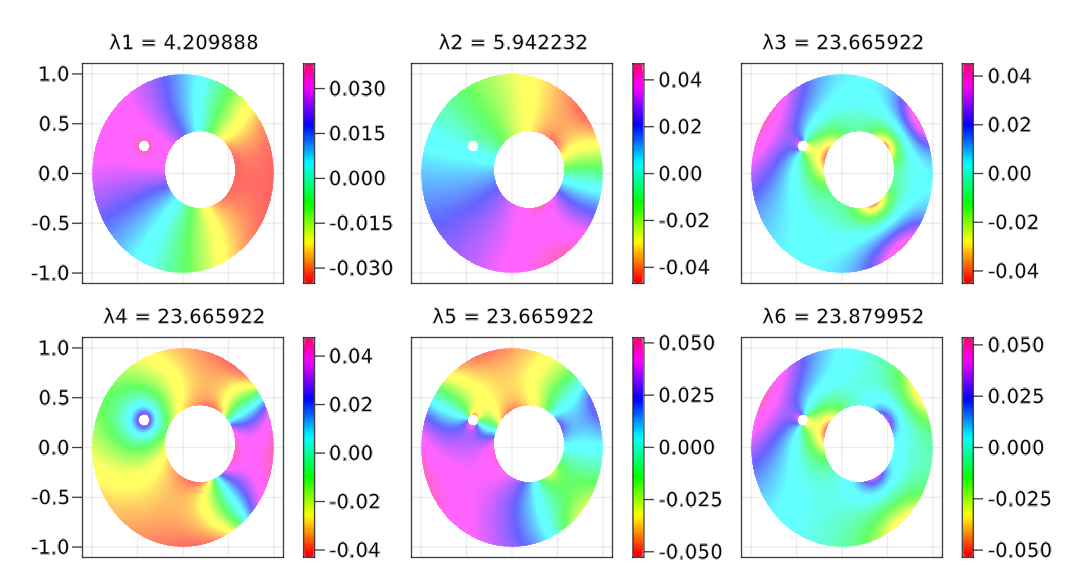} 
    \caption{Six linearly independent eigenfunctions associated to the third eigenvalue for 
             three boundary components.}
    \label{f:eigs3lambda3}
\end{figure}

\begin{figure}
    \centering
    \includegraphics[width=\textwidth]{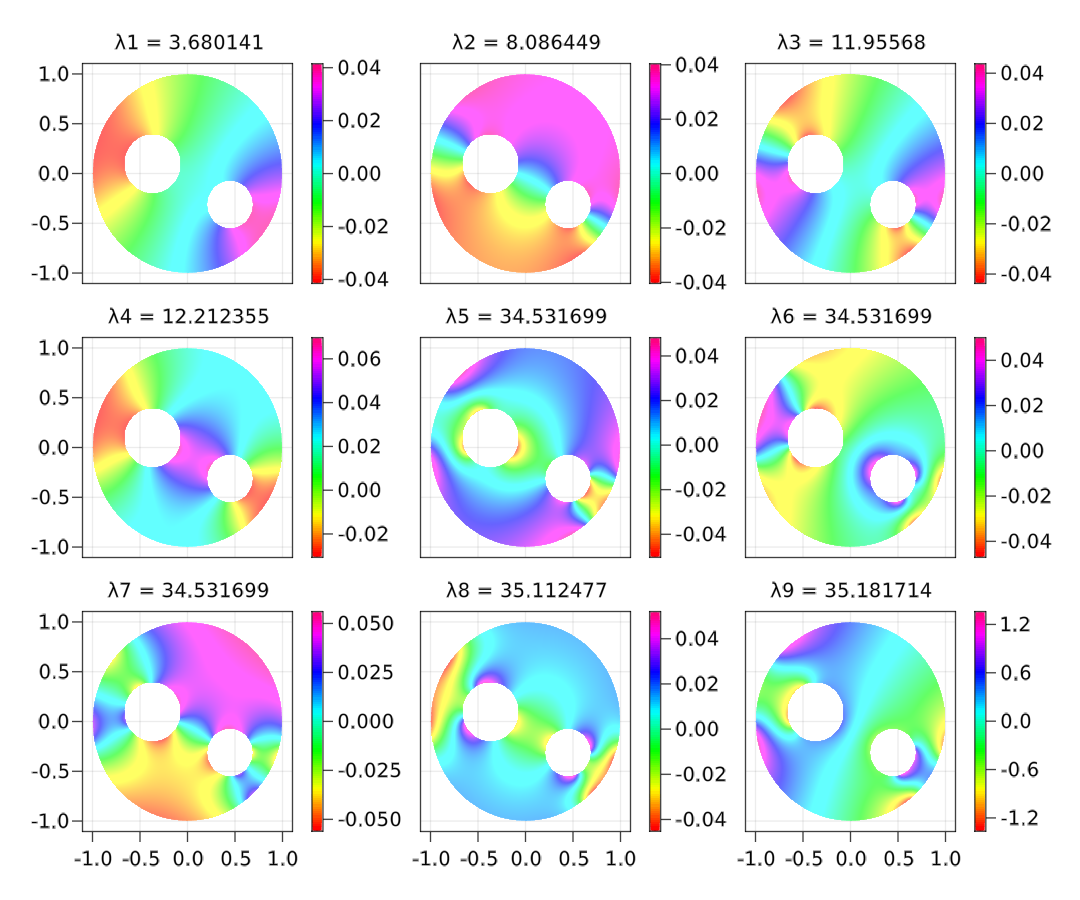} 
    \caption{Nine first linearly independent eigenfunctions associated to the fifth 
             eigenvalue for three boundary components.}
    \label{f:eigs3lambda5}
\end{figure}

\begin{figure}
    \centering
    \includegraphics[width=0.3\textwidth]{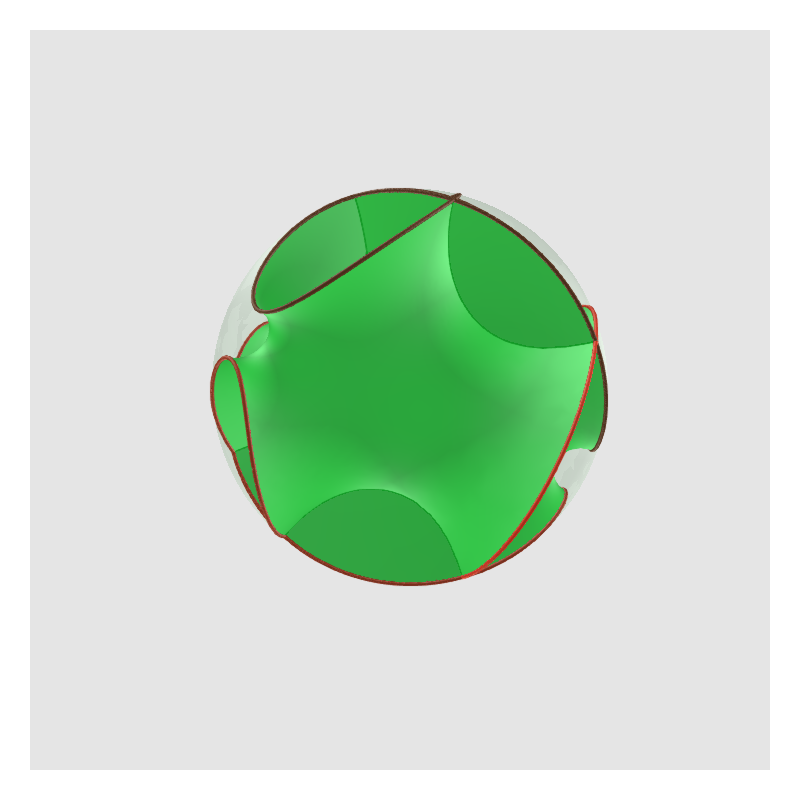} 
    \includegraphics[width=0.3\textwidth]{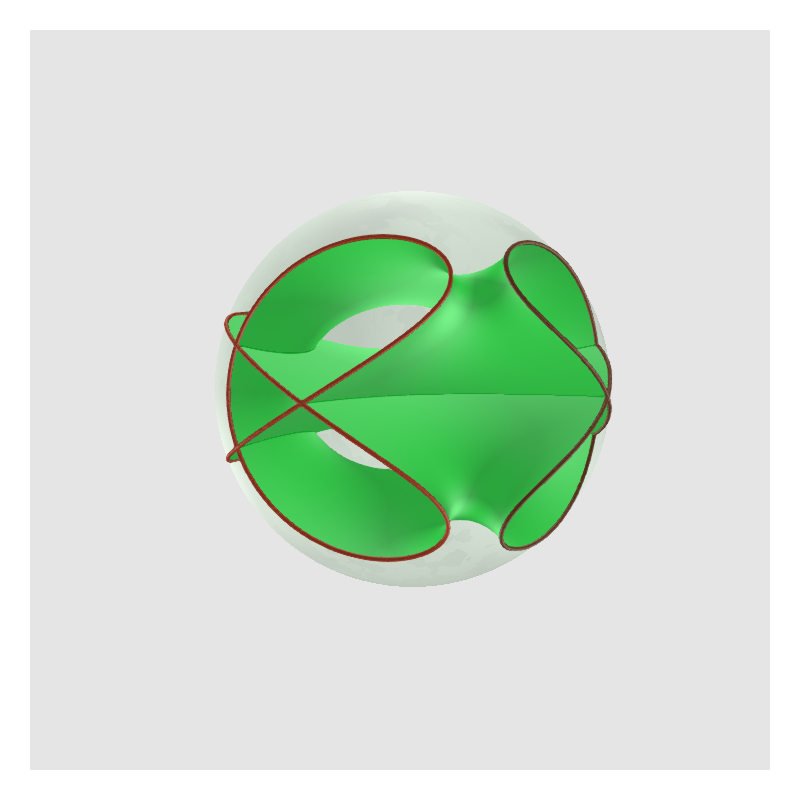} 
    \includegraphics[width=0.3\textwidth]{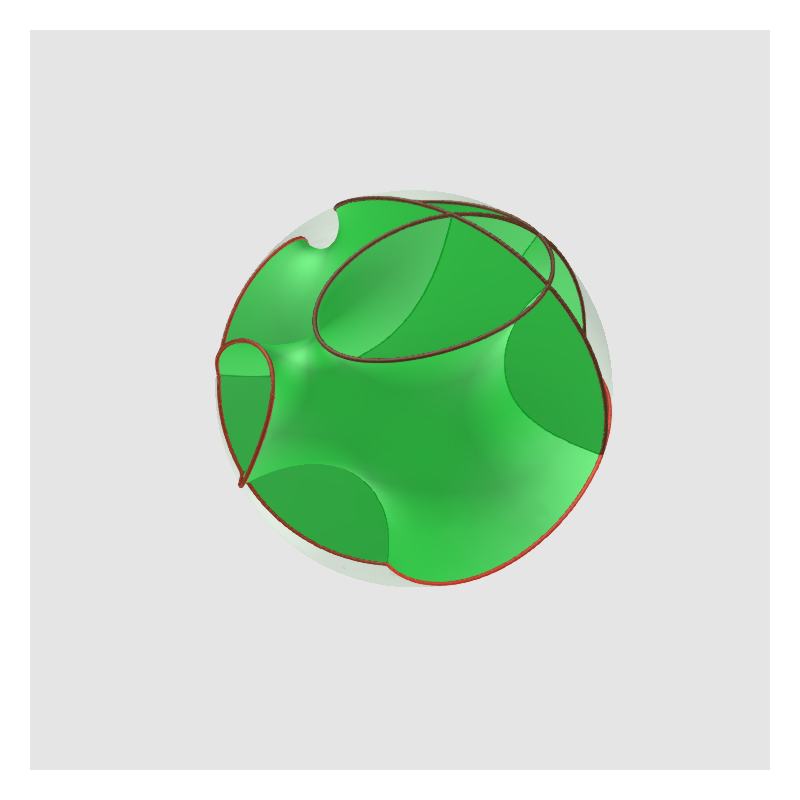}  \\
    \includegraphics[width=0.3\textwidth]{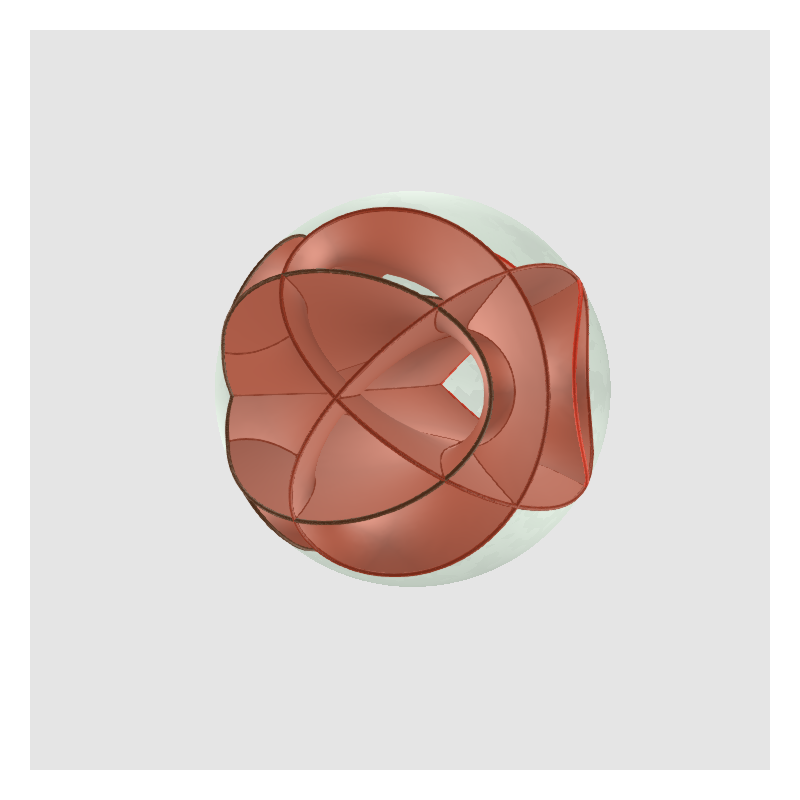} 
    \includegraphics[width=0.3\textwidth]{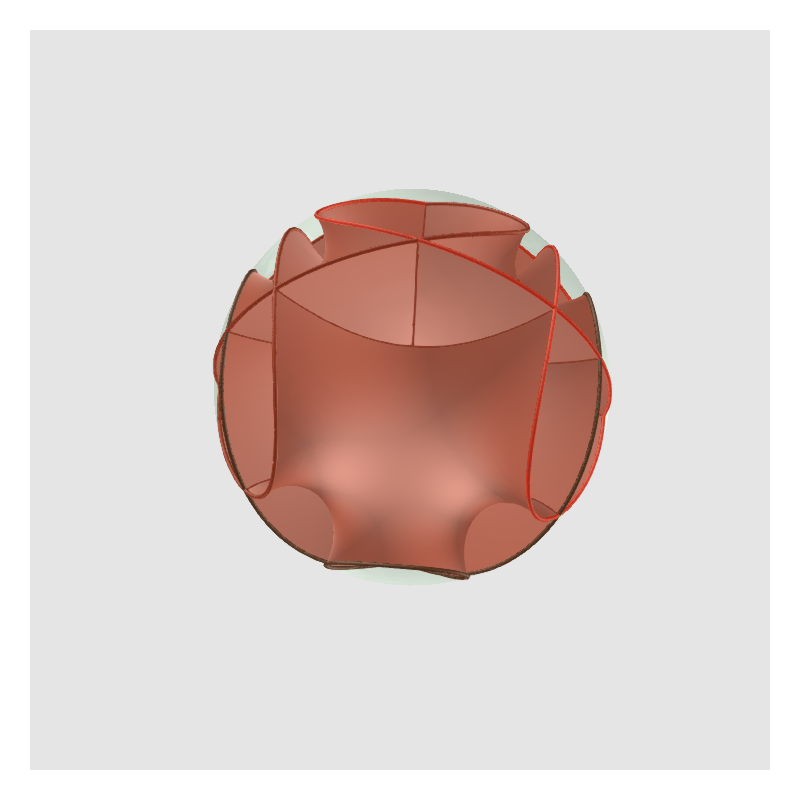} 
    \includegraphics[width=0.3\textwidth]{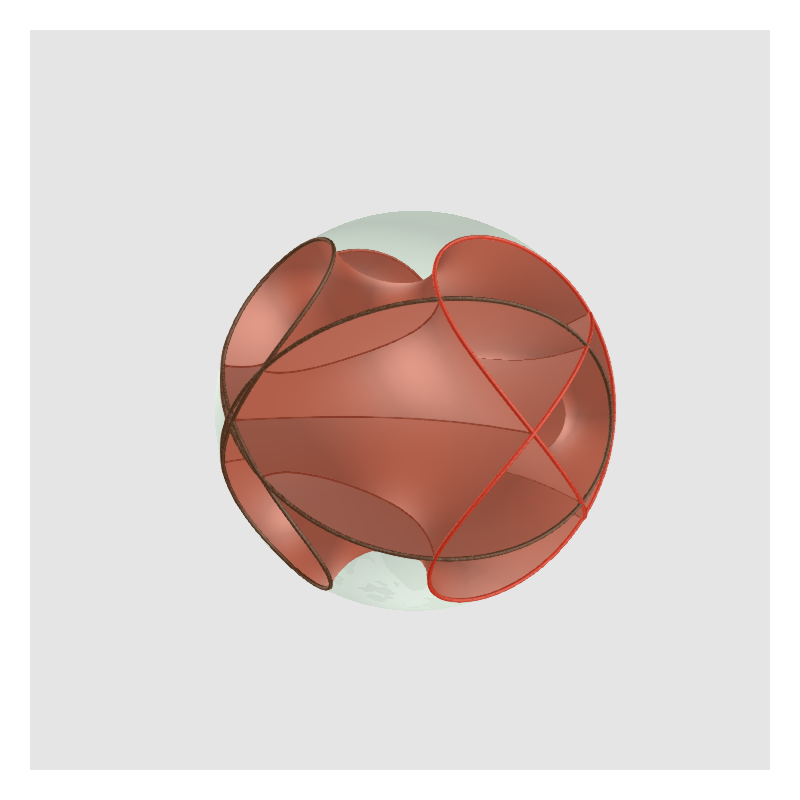}  

    \caption{Two distinct approximations of a minimal surface in the ball with 
     three connected components of the boundary associated to the third and fifth 
     Steklov eigenvalues.}
    \label{f:surflambda35}
\end{figure}

\begin{table}[t]
\begin{center}
\begin{tabular}{ l | c | c | c }
b & $\tilde \sigma_{1}$ & compare to \cite{Girouard2020} & BC center configuration \\
\hline 
2 & $10.4748$ &  & (critical catenoid) Digon \\
3 & $12.0120$ &  & equilateral triangle \\
4& $13.6676$ & $4.3505\pi\approx13.6675$ & regular tetrahedron \\
5 & $14.4687$ &  & triangular bipyramid \\
6 & $15.4292$ & $4.9099\pi\approx15.4249$ & regular octahedron \\ 
7 & $15.9520$ &  & pentagonal bipyramid \\ 
8& $16.4954$ & $5.2282\pi\approx16.4249^*$ & square antiprism (not regular) \\
9& $16.9707$ &  & triaugmented triangular prism \\
12& $18.0687$ & $5.7514\pi\approx18.0686$ & regular icosahedron \\
15 & $18.7934$ &  &  triangular symmetry \\
20 &  $19.7076$ & $6.2299\pi\approx19.5718^*$ & irregular, not dodecahedron
\end{tabular}
\end{center}
\label{t:valComp1}
\caption{For different number of boundary components $b$, we report 
the value of  the first nontrivial normalized Steklov eigenvalue $\tilde \sigma_1 = \sigma_1 L$, 
the value obtained by \cite{Girouard2020}, 
and the configuration of the centers of the boundary components. 
For $b=8$ and $b=20$, our configuration of boundary components differs from \cite{Girouard2020}, so the values should not be directly compared (indicated with an asterisk). }
\end{table}%

\section{Numerical solutions of the extremal Steklov eigenvalue problem and the corresponding free boundary minimal surfaces} \label{s:NumResults}
In this section, we describe the solutions for the extremal Steklov eigenvalue problem \eqref{e:MaxSteklov2}, 
for various number of boundary components (BC), $b$, and eigenvalue number, $j$,
and the  corresponding free boundary minimal surfaces (FBMS). 

\subsection{First nontrivial eigenvalue ($k=1$)}
We first consider the first nontrivial eigenvalue ($k=1$) for varying numbers of  BC, $b=2,\ldots,9, 12, 15, 20$. 
In each case, the multiplicity of the extremal eigenvalue is three, as expected \cite{Fraser_2015}. 
In Figure~\ref{f:alldisks}, we plot the optimal punctured disks, $\Omega_{c,r}$, for $b=2,\ldots,9$ and $b=12$ BC. 
In Figures~\ref{f:eigs23} and \ref{f:eigs45}, we plot three linearly independent eigenfunctions associated to the first eigenvalue on their respective punctured disk for $b=2,3,4,5$ BC.
For these values of $b$, the corresponding optimal densities are plotted in Figures~\ref{f:rho23}, \ref{f:rho4}, and \ref{f:rho5}.
In Figures~\ref{f:surf34} and \ref{f:surf51215}, we plot the corresponding (approximate) FBMS in the ball for $b=3,4,5,12,15$ BC.
In all cases, the BC of the FBMS are positioned at very symmetric locations, as further illustrated in Figure~\ref{f:platonic}. 
Values of $\tilde \sigma_{1}$ and additional information about these configurations are recorded in Table~\ref{t:valComp1}.
Additional figures, including gifs, can be found at \cite{EdouardWebpage} and were not included here for brevity. 

We now make a few more detailed remarks for the problem with the various number of  BC,  $b$, considered, 
especially for values of $b$ that are related to the platonic solids. 
For some values of $b$, we also compare to the FBMS discussed in \cite{Girouard2020}. 

For $b=2$, we recover the critical catenoid, the known FBMS \cite{Fraser_2015} that we also discussed in Section~\ref{s:CC}. 
Note that in Figure~\ref{f:alldisks} the hole is centered within the disk and in Figure~\ref{f:rho23}, the density is constant on each BC. 
The eigenfunctions plotted in Figure~\ref{f:eigs23} exhibit symmetries and are explicitly given in Section~\ref{s:CC}; see Figure~\ref{f:StekEigCenAnn}(right). 

For $b=3$, the FBMS  has BC positioned with centers on an equilateral triangle inscribed on a great circle of the sphere; see Figure~\ref{f:surf34}. 
Interestingly, the holes in the domain, $\Omega_{c,r}$, are slightly asymmetrically configured; see Figure~\ref{f:alldisks}. 
The densities plotted in Figure~\ref{f:rho23} do not exhibit symmetry.
The eigenfunctions plotted in Figure~\ref{f:eigs23} do not exhibit symmetries, but this could be a result of our (arbitrary) choice within the three dimensional eigenspace. 

For $b=4$, the FBMS has BC positioned with centers at the vertices of a regular tetrahedron; see Figure~\ref{f:surf34}. 
This is further illustrated in Figure~\ref{f:platonic}, where the BC are overlaid on a regular tetrahedron. 
 A similar minimal surface was computed in  \cite{Girouard2020} and the value of $\tilde \sigma_1$ is within $10^{-4}$; see Table~\ref{t:valComp1}. 
In Figure~\ref{f:alldisks}, the holes in the domain, $\Omega_{c,r}$, are slightly asymmetrically configured. 
In Figure~\ref{f:rho4}, the density on the outer boundary is nearly constant and the densities on the inner boundaries are similar to each other. 
There is no clear structure to the eigenfunctions potted in Figure~\ref{f:eigs45}. 

For $b=5$, the FBMS has BC positioned with centers at the vertices of a triangular bipyramid; see Figure~\ref{f:surf51215}. 
In Figure~\ref{f:alldisks}, the holes in the domain, $\Omega_{c,r}$, are not only asymmetrically configured, but the radii of the holes vary. 
In Figure~\ref{f:rho4}, the density on the outer boundary is nearly constant and the densities on the inner boundaries are similar to each other. 
Again, the eigenfunctions plotted in Figure~\ref{f:eigs45} do not appear to be structured. 

For $b=6$, the FBMS has BC positioned with centers at the vertices of a regular octahedron; see Figure~\ref{f:surf51215}. 
This is further illustrated in Figure~\ref{f:platonic}, where the BC are overlaid on a regular octahedron. 
Again, a similar minimal surface was computed in  \cite{Girouard2020} and the value of $\tilde \sigma_1$ is within $5 \times 10^{-3}$; see Table~\ref{t:valComp1}. 
In Figure~\ref{f:alldisks}, the holes in the domain, $\Omega_{c,r}$, are  slightly asymmetrically configured; 
there is a small hole near the origin and four holes of equal radii roughly centered at the vertices of a square. 
In Figure~\ref{f:rho4}, the density on the outer boundary is nearly constant and the densities on the inner boundaries are similar to each other. 

For $b=7$, the FBMS has BC  positioned at the vertices of a pentagonal bipyramid. Figures of the FBMS can be found at \cite{EdouardWebpage}. 
In Figure~\ref{f:alldisks}, the domain, $\Omega_{c,r}$, has a small (uncentered) hole surrounded by five holes. 

For $b=8$, the FBMS has BC positioned  at the vertices of a square antiprism; see \cite{EdouardWebpage} and Figure~\ref{f:platonic}. 
Interestingly, we obtain $\tilde \sigma_1 \approx 16.4954$ for this surface, which is larger than the value obtained for the FBMS with BC at the vertices of a cube, as discussed in \cite{Girouard2020}, with value $\tilde \sigma_1 \approx 16.4249$; see Table~\ref{t:valComp1}. 
In Figure~\ref{f:alldisks}, the domain, $\Omega_{c,r}$, has three smaller holes surrounded by four larger holes. 

For $b=9$, the FBMS has BC positioned at the vertices of a triaugmented triangular prism. Figures of the FBMS can be found at \cite{EdouardWebpage}. 
In Figure~\ref{f:alldisks}, the domain, $\Omega_{c,r}$, has three smaller holes surrounded by five larger holes. 

For $b= 12$, the FBMS has BC positioned at the vertices of a regular icosahedron; see Figure~\ref{f:surf51215}. 
This is further illustrated in Figure~\ref{f:platonic}, where the BC are overlaid with a regular icosahedron. 
A similar minimal surface was computed in  \cite{Girouard2020} and the value of $\tilde \sigma_1$ is within $10^{-4}$; see Table~\ref{t:valComp1}. 
In Figure~\ref{f:alldisks}, the domain, $\Omega_{c,r}$, have one small uncentered hole, surrounded by five medium-sized holes, surrounded by five larger holes. 

For $b=15$ the FBMS is plotted in Figure~\ref{f:surf51215}. The FBMS has BC that are positioned with centers with triangular symmetry. 

For $b=20$ the FBMS has irregularly located BC; a figure can be found at \cite{EdouardWebpage}. 
Interestingly, we obtain $\tilde \sigma_1 \approx 19.7076$ for this surface, which is larger than the value obtained for the FBMS with BC at the vertices of a regular dodecahedron, as discussed in \cite{Girouard2020}, with value $\tilde \sigma_1 \approx 19.57189$; see Table~\ref{t:valComp1}. 

\medskip

We have observed that the FBMS for $b=8$ and $20$ do not have BCs centered at the vertices of a platonic solid. 
It seems that the positions of the BCs are related to the minimizing configurations for Thompson's problem; known as the Fekete points \cite{fekete1923verteilung, brownmin}. 

We note that the FBMS obtained here are closely related to the $k$-noid surfaces; see \cite{Bloomington}. It may be appropriate to the FBMS computed here as \emph{critical $k$-noids}.

\subsection{Higher eigenvalues ($j\geq 2$)} 
Here, we consider the extremal Steklov eigenvalue problem \eqref{e:MaxSteklov2}, for higher eigenvalues, $\tilde \sigma_j$, $j\geq 2$. 
Less in known in this case and, in particular, the multiplicity of the optimal eigenvalue, and hence the dimension in which the FBMS exists, is unknown. 

We recall from  \cite{Fan_2014} (see also Section~\ref{s:CC}) that by maximizing $\sigma_j$ for odd $j$ among  rotationally symmetric annuli yields an $\frac{j+1}{2}$ covering of the critical catenoid, a FBMS with $b=2$ boundary components and $j$-th normalized Steklov eigenvalue,  
$$
\tilde \sigma_j = \frac{j+1}{2} \tilde \sigma_1^\star, 
\qquad \qquad j \ \textrm{odd}. 
$$
We also recall the result of \cite[Theorem 5.3]{Fraser2019}, that the degenerate surface consisting of  the critical catenoid glued to $j -1$ unit disks, 
is a FBMS with $b=2$ boundary components  in $3+ 2(j-1)$ dimensions with $j$-th normalized Steklov eigenvalue, 
$$
\tilde \sigma_j = \tilde \sigma_1 + (j-1) 2 \pi. 
$$

We first consider $b=2$ BC and eigenvalue $j=2$. In 
this case, the density  $\rho$ on the outer boundary of the punctured disk becomes degenerate and resembles the $\rho$ 
discussed in Example~\ref{e:ex3} and displayed in Figure~\ref{f:critical_caternoid_glue_disk_to_a_disk}.
We believe that this $\rho$ corresponds to the critical catenoid glued to a disc, but this is difficult to resolve using our numerical method; see Example~\ref{e:ex5}. For other higher eigenvalues, we see similar phenomena for some initializations of  $\rho$. However, there are a few values of eigenvalue number $j$ and BC $b$, that give interesting local maximizers and are very robust with respect to the initialization. 

For $b=2$ BC and   $j=3$ eigenvalue,  we obtain a double covering of the critical catenoid as obtained by \cite{Fan_2014}; 
see \cite{EdouardWebpage}.   
The value obtained is $\tilde \sigma_3 = 2 \tilde \sigma_1^* \approx 20.9496$.
This is a local maximizer \cite[Theorem 5.3]{Fraser2019}; 
we can obtain the value $\tilde \sigma_j = \tilde \sigma_1^* + 4 \pi \approx 23.0412$  by gluing a critical catenoid to two disks.

For $b=3$ BC the FBMS obtained by maximizing the $j=3$ and $j=5$ eigenvalues are displayed in  Figure~\ref{f:surflambda35}. 
If Figures~\ref{f:eigs3lambda3} and \ref{f:eigs3lambda5}, the first few eigenfunctions are plotted in the optimal domains, $\Omega_{c,r}$. 
The eigenvalues obtained are 
$\tilde \sigma_3 = 23.6659$ 
and 
$\tilde \sigma_5 = 34.5317$. 
Note that, again, these are local maximizers since larger eigenvalues can be obtained by gluing two or four balls to the surface attained by maximizing the first eigenvalue with $b=3$ BC, to obtain eigenvalues 
$\tilde \sigma_3 = 12.0120 + 2 \cdot 2 \cdot \pi \approx 24.5784$
and
$\tilde \sigma_5 = 12.0120 + 2 \cdot 4 \cdot  \pi  \approx 37.1447$.

\section{Discussion} \label{s:Disc}

In this paper, we developed computational methods to maximize 
the length-normalized $j$-th Steklov eigenvalue, 
$\tilde \sigma_j(\Sigma, g) := \sigma_j(\Sigma, g)  L(\partial  \Sigma,g)$ 
over the class of smooth Riemannian metrics, $g$ on a compact surface, $\Sigma$, with genus $\gamma$ and $b$ boundary components. 
Our numerical method involves 
(i) using conformal uniformization of multiply connected domains to avoid explicit parameterization for the class of metrics,  
(ii) accurately solving a boundary-weighted Steklov eigenvalue problem in multi-connected domains, and
(iii) developing gradient-based optimization methods for this non-smooth eigenvalue optimization problem. 
Using the connection due to Fraser and Schoen \cite{Fraser_2015}, the solutions to this
extremal Steklov eigenvalue problem  for various values of $b$ boundary components are used to generate free boundary minimal surfaces. 

In hindsight, it may have been better to perform these computations on a punctured sphere rather than a punctured disk, as a punctured disk distinguishes one boundary (the `outer' one). In particular, by considering a punctured sphere, it may be that the holes appear more symmetrically than for a punctured disk; see Figure~\ref{f:alldisks}. 

Beyond further exploring higher eigenvalues $j$ and higher numbers of boundary components $b$, there are a number of interesting extensions of this work. In particular, we would be very interested to compute extremal Steklov eigenvalues on the M\"obius band, torus, and other higher genus surfaces and use the associated eigenfunctions to generate free boundary minimal surfaces. 
We're also interested in related extremal eigenvalue problems, involving convex combinations of Steklov eigenvalues or Steklov eigenvalues for the $p$-Laplacian.

\subsection*{Acknowledgements}
The authors would like to thank the Mathematics Division, National Center of Theoretical Sciences, Taipei, Taiwan for hosting a research pair program during June 15-June 30, 2019 to support this project. 
Chiu-Yen Kao acknowledges partial support from NSF DMS 1818948.  
Braxton Osting acknowledges partial support from NSF DMS 17-52202.  
\'Edouard Oudet  acknowledges partial support from CoMeDiC (ANR-15-CE40-0006) and ShapO (ANR-18-CE40-0013). 
The authors would also like to thank Bruno Colbois, Joel Dahne, Baptiste Devyver, Alexandre Girouard,  Mikhail Karpukhin, Jean Lagace and Iosif Polterovich for useful conversations.


\bibliographystyle{alphaurl}
\bibliography{refs.bib}



\end{document}